\documentclass[11pt,reqno]{article}
\usepackage{amssymb, amsmath, amsthm, amsfonts, amscd, epsfig, subfig, gensymb}
\usepackage[dvipsnames]{xcolor}
\usepackage[utf8]{inputenc}
\usepackage[english]{babel}
\usepackage{bm}
\usepackage{bbm}
\usepackage{ulem}
\usepackage{graphicx}
\usepackage{geometry,graphicx,pgfplots, multirow}

\usepackage[export]{adjustbox}
\usepackage[titletoc,toc,title]{appendix}
\usepackage{algorithm}
\usepackage{mathtools}

\usepackage{afterpage}
\usepackage{comment}

\newtheorem{theorem}{Theorem}[section]

\newtheorem{lemma}[theorem]{Lemma}

\newtheorem{remark}{Remark}

\def\Z{\mathbb{Z}}

\def\R{\mathbb{R}}

\newcommand{\Om}{\Omega}

\newcommand{\RR}{\mathbb{R}}

\newcommand{\eps}{{\varepsilon}}

\newcommand{\p}{\partial}

\newcommand{\eqnref}[1]{(\ref {#1})}

\newcommand{\beq}{\begin{equation}}
\newcommand{\eeq}{\end{equation}}
\newcommand{\RN}[1]{%
  \textup{\uppercase\expandafter{\romannumeral#1}}%
}

\numberwithin{equation}{section}
\numberwithin{figure}{section}

\begin{document}
\title{
Bayesian optimization approach for tracking a moving target from far-field data in three dimensions
}

\date{}

\author{
Woojoo Lee\thanks{Department of Mathematical Sciences, Korea Advanced Institute of Science and Technology, 291 Daehak-ro, Yuseong-gu, Daejeon 34141, Republic of Korea (woojoo.lee@kaist.ac.kr, mklim@kaist.ac.kr).}\and
Sangwoo Kang\thanks{Graduate School of Data Science, Pusan National University, 2, Busandaehak-ro 63beon-gil, Geumjeong-gu, Busan 46241, Republic of Korea (sangwoo.kang@pusan.ac.kr)}
\and Mikyoung Lim\footnotemark[1]
}

\maketitle

\begin{abstract}
We investigate a three-dimensional inverse scattering problem for tracking a rigidly moving target from far-field data generated by a single incident field. Extending our recent two-dimensional study, we develop a Bayesian optimization framework for simultaneously tracking the target's location and orientation over successive time steps, with the translational and rotational motions modeled as independent stochastic processes. We derive analytical formulas for the far-field pattern under translations and rotations and use them to design a Bayesian optimization procedure tailored to the tracking problem. We further establish posterior consistency for the underlying probabilistic model. When the target shape is unknown, its shape is identified at the initial time using a fully connected neural network trained on a precomputed dataset. Numerical experiments validate the effectiveness of the proposed framework.
\end{abstract}




\section{Introduction}
Inverse scattering problems aim to identify the geometric and material features of targets and have been extensively studied across a wide range of applications, including biomedical imaging \cite{Abubakar:2002:IBD}, nondestructive testing \cite{Marklein:2006:ISI, Salucci:2016:RTN}, and remote sensing \cite{Woodhouse:2017:IMR}. The primary challenge lies in tackling the inherent nonlinearity and ill-posedness of inverse problems. Numerous approaches have been developed and can be classified into iterative, decomposition, and sampling methods, as discussed in the survey \cite{Luke:2003:NRT}. Extensive investigations have also been conducted on piecewise homogeneous media \cite{Liu:2010:DIO} as well as fully inhomogeneous media \cite{Colton:1998:IAE,Kirsch:1998:ROI,Li:2013:DSM}. Furthermore, recent research has focused on accommodating restrictive measurement configurations, such as limited-view, monostatic, and bistatic measurements \cite{Huang:2022:BAI, Ning:2025:CPF, Kang:2022:MSM, Kang:2022:FIS}.

For moving targets, the primary interest shifts to estimating trajectories, a process commonly referred to as {\textit{tracking}}. 
The efficient tracking of moving objects is crucial for practical applications, including autonomous driving \cite{Leon:2021:RTT}, robotics \cite{Robin:2016:MRT}, and radar imaging \cite{Wang:2012:MRI, Haworth:2007:DTM}. To overcome the intrinsic difficulties posed by unpredictable target trajectories, various mathematical frameworks have been proposed. These tracking strategies range from algebraic methods \cite{Ohe:2011:RRT, Ohe:2020:RRM} and Bayesian inference \cite{Wang:2023:LMS} for point sources to linearized imaging theories and Kirchhoff migration \cite{Son:2024:RTT} for moving objects. Moreover, multipole expansions have been utilized in the context of conductivity problems \cite{Ammari:2013:TMT}.

\begin{figure}[H]
	\centering
	\includegraphics[width=0.4\textwidth]{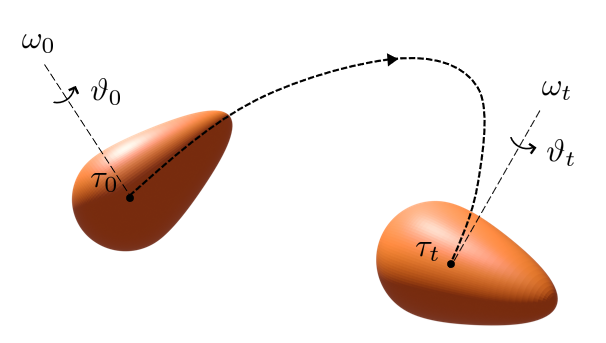}
		\caption{\label{fig:tracking} Target tracking. At each time step $t$, we reconstruct the target's location $\tau_t$ and orientation (characterized by a rotation angle $\vartheta_t$ around an axis $\omega_t$) from far-field data using a single incident field. The translational motion and rotational orientation are assumed to be independent.}
\end{figure}

In this paper, we investigate the three-dimensional problem of tracking the location and orientation of a moving target, extending our recent Bayesian optimization approach in two dimensions \cite{Lee:2025:BOA}. We consider a sound-soft scatterer undergoing rigid motion, whose shape is preserved while its location and orientation evolve over time; see Figure \ref{fig:tracking}. In particular, we allow the translational and rotational motions to evolve independently. This makes the tracking problem substantially more challenging than in settings where translation and rotation are coupled. We use far-field data generated by a single incident field under both full- and limited-aperture measurement configurations. 

Let $\Om\subset\RR^3$ denote the reference target domain, and let $\Om_t$ denote its configuration at time $t\geq 0$, with $\Om_0=\Om$. We assume that $\Omega $ is a bounded, simply connected domain with a $C^2$ boundary.
For a fixed wavenumber $k$, the total field $u=u^i+u^s$ associated with the sound-soft target $\Om$ satisfies 
\begin{align} 
\begin{cases} \label{prob:scattering}
\Delta u + k^2 u=0 & \text{in } \ \R^3\backslash \overline{\Omega},\\
u=0& \text{on } \ \p \Omega,\\
\lim_{r\to\infty} r\Big(\dfrac{\p u^s}{\p r}-iku^s\Big)=0 & \mbox{with }r=|x|.
\end{cases}
\end{align}
Here, $u^i$ is the incident plane wave
$$u^i(x)=e^{ikx\cdot d},$$ 
where $d\in S^2 := \{ y \in \mathbb{R}^3\,|\, \|y\| = 1 \}$ is the incident direction, and $u^s$ is the scattered field. 
The scattered field admits the far-field expansion
\begin{equation}\label{def:ff}
u^{s}(x) = \frac{e^{ik|x|}}{|x|}
\left[ u^{\infty}(\hat{x}) + O\!\left(\frac{1}{|x|}\right) \right]\quad\mbox{as }|x| \to \infty,
\end{equation}
uniformly with respect to the observation directions $\hat{x} = x/|x| \in S^2$; see \cite{Colton:1998:IAE}.

Our objective is to recover the location and orientation of $\Om_t$ at discrete time steps from its far-field measurements. 
Although the target undergoes continuous motion, we treat the inverse scattering problem at each time step as a static reconstruction task. This approach leverages the fact that the acquisition time for each measurement is sufficiently small compared to the time scale of the target's motion.
We assume the far-field data, $u^\infty(\hat{x})$, are acquired using a single fixed incident field under either full- or limited-aperture observation. Based on these measurements, we develop a method to track the moving target by reconstructing its complete target state (location and orientation) at each time step. 

For this inverse scattering problem, a natural approach is to minimize, at each time step, a discrepancy between the measured far-field data and the far-field data computed for a candidate target configuration. Since each evaluation of this objective function generally requires solving a forward scattering problem to compute the far-field pattern, the resulting optimization problem is computationally expensive. 

Bayesian optimization provides an efficient framework for the global optimization of black-box or computationally expensive objective function for which gradient information is unavailable or costly to obtain. It constructs a probabilistic surrogate model for the objective function and updates this model sequentially from a prior to a posterior distribution as function evaluations are collected. The surrogate model guides the selection of subsequent evaluation points through an acquisition function, which rigorously manages the trade-off between probing unexplored regions (exploration) and exploiting known promising regions (exploitation). Gaussian process priors are widely used in this setting because they yield analytically tractable closed-form expressions for the posterior distribution \cite{OHagan:1978:CFO,Mockus:1994:ABA}. Additionally, Bayesian formulation enables uncertainty quantification, which could be valuable for inverse scattering problems involving noisy data \cite{Carpio:2020:BAI, Huang:2022:BAI}. A general collective theory for Bayesian inverse problems is provided in \cite{Dashti:2017:BAI}. We refer to the related literature \cite{Huang:2021:BOF, Vargas-Hernandes:2019:BOI, Hammerschmidt:2018:SIP} for further applications of Bayesian optimization to inverse problems. 

Our previous work \cite{Lee:2025:BOA} introduced a two-dimensional Bayesian optimization framework for tracking the location and orientation of a moving target from far-field data.
The method leverages analytical properties of the far-field pattern under rigid motion to reduce the cost of evaluating a discrepancy-based objective function.
For targets of unknown shape, the reference shape is identified at the initial time using a fully connected neural network (FCNN), after which tracking proceeds without additional shape training data.

The present work develops a method for tracking a moving target in three dimensions, building on the two-dimensional framework in \cite{Lee:2025:BOA}. In contrast to the planar setting, a rigidly moving target in three dimensions has six degrees of freedom---three for translation and three for rotation---which makes the tracking problem considerably more challenging. 
Our main contributions are as follows. First, we derive three-dimensional counterparts of the translation and rotation formulas for the far-field pattern and establish its local Lipschitz dependence on the orientation near the true orientation angle. 
Second, we incorporate these analytical properties of the far-field pattern into a Bayesian optimization procedure. This significantly reduces the computational cost of evaluating the objective function. Moreover, we design analytically motivated priors and analyze the selection behavior of the acquisition function, neither of which was addressed in \cite{Lee:2025:BOA}.
Finally, we establish posterior consistency for the probabilistic model of the objective function underlying the proposed optimization procedure.

When the target shape is unknown, we employ a simple FCNN to identify the reference shape from the far-field data acquired at the initial time. The target is modeled as a perturbed ellipsoid parametrized by ellipsoidal harmonics. Its boundary is described by a finite number of shape parameters, and the FCNN is trained to learn the map from the initial far-field data to these parameters. The training dataset is used only for the initial shape-identification step. Once the reference shape has been identified, the subsequent tracking procedure relies solely on far-field data acquired from the moving target and requires no additional training data. 
It is worth noting that recent studies have introduced deep neural network architectures specialized for inverse scattering problems; see, for example, \cite{Guo:2022:PED, Chen:2020:ROD}. For a survey of data-driven methods for inverse problems, including deep neural networks, we refer to \cite{Arridge:2019:SIP}.

The proposed framework accommodates targets of general shape whose motion may be arbitrary. We test the proposed method using synthetic far-field data with additive random noise. The numerical results demonstrate that the method performs well even for an extended target under limited-aperture measurement data, provided that the shape of the target is known {\it a priori} or can be accurately identified. 
We note that sampling-type methods, despite being computationally fast and non-iterative, may be less reliable in such configurations \cite{Son:2024:RTT}; see also \cite{Kang:2022:MSM,Ning:2025:CPF}.

The paper is organized as follows. 
Section \ref{sec:ff} derives analytical properties of the far-field pattern under translations and rotations. Section \ref{sec:bo} develops a Bayesian optimization-based method for recovering the target configuration, and establishes posterior consistency for the underlying posterior model. Section \ref{sec:framework} presents the complete framework for tracking a moving target, including shape identification and the recovery of location and orientation. Numerical simulations are provided in Section \ref{sec:num}. Finally, Section \ref{sec:fin} concludes the paper.

\section{Far-field patterns under rigid motion} \label{sec:ff}
We first review the layer potential technique for the scattering problem (Section \ref{subsec:ff:lp}). We then rigorously derive general properties (Section \ref{subsec:ff:rigid}) and establish local regularity results with respect to the orientation angle (Section \ref{subsec:ff:stab}) for the far-field patterns of a target undergoing rigid motions in three dimensions. 

\subsection{Layer potential technique}\label{subsec:ff:lp}
Let $\Omega \subset \mathbb{R}^3$ be a bounded, simply connected domain with $C^2$ boundary, and let $C(\partial\Omega)$ denote the space of continuous functions on $\partial\Omega$. 

The fundamental solution of the Helmholtz equation in $\RR^3$ is given by
\[
\Phi(x,y) := \frac{1}{4\pi}\frac{e^{ik|x-y|}}{|x-y|}.
\]
The single-layer operator $S_\Omega : C(\partial\Omega)\to C(\partial\Omega)$ and the double-layer operator $K_\Omega : C(\partial\Omega)\to C(\partial\Omega)$ are defined, respectively, by
\begin{align*}
(S_\Omega\phi)(x)& = \int_{\partial\Omega} \Phi(x,y)\phi(y)\,ds(y), \quad x\in\p\Om,\\
(K_\Omega\phi)(x)& = \int_{\partial\Omega} \frac{\partial \Phi(x,y)}{\partial n(y)}\phi(y)\,ds(y), \quad x\in\p\Om,
\end{align*}
where $\phi$ is a density function on $\p\Om$ and $n(y)$ denotes the outward unit normal to $\p\Om$ at $y \in \partial\Omega$.

Let $u^i$ and $u^s$ denote the incident and scattered fields in \eqnref{prob:scattering}, respectively. Then $u^s$ admits a boundary integral representation. More precisely, by Green's formula for the Helmholtz equation and the properties of the acoustic single- and double-layer potentials, there exist a density function $\phi$ on $\p\Om$ and a real coupling parameter $\eta \neq 0$ such that (see \cite{Colton:1998:IAE})
\begin{align}\label{eqn:u^s:orig}
u^s(x)
= \int_{\partial\Omega}
\bigg[
\frac{\partial \Phi(x,y)}{\partial n(y)} - i\eta\,\Phi(x,y)
\bigg]\phi(y)\,ds(y),
\quad x \in \mathbb{R}^3 \setminus \overline{\Omega}. 
\end{align}
Applying the jump relations for layer potentials to \eqref{eqn:u^s:orig} yields the boundary relation 
\begin{align}\label{eqn:u^s:lp}
u^s = \left(\frac{1}{2}I + K_\Omega - i\eta S_\Omega\right)\phi
\quad \mbox{on }\partial{\Omega}.
\end{align}
Imposing the boundary condition $u=0$ on $\p\Om$ in \eqref{prob:scattering}, we obtain the boundary integral equation
\begin{align}\label{eqn:density}
\left(\frac{1}{2}I + K_\Omega - i\eta S_\Omega\right)\phi = -u^i
\quad \text{on } \partial\Omega,
\end{align}
where $u^i = e^{ik x \cdot d}$.
It is known that \eqref{eqn:density} admits a unique solution $\phi \in C(\partial\Omega)$ (see \cite[Theorem~3.11]{Colton:1998:IAE}).
Thus, for a prescribed incident field $u^i$, the scattered field $u^s$ is obtained by first solving \eqref{eqn:density} for $\phi$, and then substituting $\phi$ into the representation formula \eqref{eqn:u^s:lp}.

Taking the limit $|x| \to \infty$ in \eqref{eqn:u^s:orig}, we obtain
\begin{align}
u^s(x) &= \frac{1}{4\pi}\int_{\partial\Omega}
\bigg[ \frac{\partial}{\partial n(y)}\frac{e^{ik|x-y|}}{|x-y|} - i\eta\,\frac{e^{ik|x-y|}}{|x-y|}\bigg]\phi(y)\,ds(y) \notag \\
&= \frac{e^{ik|x|}}{|x|}\left(
\frac{1}{4\pi}\int_{\partial\Omega}\bigg[\frac{\partial e^{-ik\hat{x}\cdot y}}{\partial n(y)} - i\eta\, e^{-ik\hat{x}\cdot y}\bigg]\phi(y)\,ds(y)
+ O\!\left(\frac{1}{|x|}\right)\right). \notag
\end{align}
Comparing this with the far-field expansion \eqref{def:ff}, we obtain the boundary integral representation for the far-field pattern:
\begin{align}
u^{\infty}(\hat{x}) &= \frac{1}{4\pi}\int_{\partial\Omega} \bigg[\frac{\partial e^{-ik\hat{x}\cdot y}}{\partial n(y)} - i\eta\, e^{-ik\hat{x}\cdot y} \bigg]\phi(y)\,ds(y) \notag \\
&= -\frac{i}{4\pi}\int_{\partial\Omega} \big(k\,\hat{x}\cdot n(y) + \eta\big) e^{-ik\hat{x}\cdot y}\,\phi(y)\,ds(y).
\label{eqn:ff:orig}
\end{align}

We emphasize that both the far-field pattern and the associated boundary density depends on the scatterer $\Om$ and the incident direction $d$. To make this dependence explicit, we write $u_\Omega^\infty(\hat{x};d)$ for the far-field pattern and denote by $\phi_d$ the density solving \eqref{eqn:density}.

 \subsection{Transformation of far-field patterns under rigid motions} \label{subsec:ff:rigid}

Rigid motions in three-dimensions are described by compositions of translations and rotations. 
Given $\tau \in \mathbb{R}^3$, we denote the translated scatterer of $\Omega$ by
$\Omega+\tau := \{ z+\tau \mid z \in \Omega \}$. Likewise, for a rotation axis with direction $\omega \in S^2$ on the unit sphere in $\mathbb{R}^3$ and an angle $\vartheta \in (-\pi,\pi]$, we write the rotated scatterer as $R^\omega_\vartheta \Omega := \{ R^\omega_\vartheta z \mid z \in \Omega \}$, where $R^\omega_\vartheta z$ represents the rotation of $z$ by the angle $\vartheta$ about the axis with direction $\omega$. 

We model the scatterer at time $t$, denoted by $\Omega_t$, as a rigidly moved copy of the initial configuration $\Omega_0$ so that \begin{equation}\label{Omt:rigid}
\Omega_t = R^\omega_\vartheta \Omega_0 + \tau
\end{equation}
for some $\tau$, $\omega$, and $\vartheta$, where $\Omega_0$ is assumed to contain the origin.
We consider the problem of recovering the translation vector $\tau$ and the rotation parameters $\omega$ and $\vartheta$ for each $t>0$ such that \eqref{Omt:rigid} holds. This is achieved by matching far-field data, namely by seeking parameters for which
\begin{equation}\label{Omt:ff}
u^\infty_{\Omega_t}(\hat{x};d) = u^\infty_{R^\omega_\vartheta \Omega_0 + \tau}(\hat{x};d),
\end{equation} 
given a fixed incident direction $d$ and multiple observation directions $\hat{x} \in S^2$. Here $u_\Omega^\infty(\hat{x};d)$ denotes the far-field pattern at observation direction $\hat{x} \in S^2$ corresponding to the incident direction $d \in S^2$.

\begin{remark}\label{rmk:nonunique}
The rotation parameter $\vartheta$ in \eqref{Omt:ff} may fail to be uniquely determined, in particular when the scatterer exhibits rotational symmetries; see \cite[Remark~2]{Lee:2025:BOA}. To exclude this ambiguity, we restrict the rotation between consecutive time instances to a sufficiently small interval and assume that the scatterer is not invariant under rotations $R^\omega_\vartheta$ for small angles $\vartheta$.
\end{remark}

We summarize the transformation rules for the far-field pattern induced by translations and rotations of the scatterer $\Omega$ in Lemmas \ref{thm:translation} and \ref{thm:rotation}, respectively, which constitute the key ingredients of our tracking algorithm. These translation and rotation formulas extend the corresponding two-dimensional statements in \cite{Lee:2025:BOA} to three dimensions.
Since their derivations rely on layer potential arguments that are independent of the spatial dimension in \eqref{prob:scattering}, the proofs are omitted.

\begin{lemma}\label{thm:translation}
Let $\tau\in\RR^3$ and the incident field $u^i(z)=e^{ikz\cdot d}$ with direction $d\in S^2$ be fixed. The far-field pattern of the translated scatterer $\Omega+\tau$ satisfies
\begin{align} \label{eqn:translation}
u^\infty_{\Omega+\tau}(\hat x;d) \, =\, e^{-ik\tau\cdot (\hat{x}-d)}u^\infty_{\Omega}(\hat x;d). 
\end{align}
\end{lemma}

\begin{lemma}\label{thm:rotation}
Let $\vartheta\in \R$, $\omega \in S^2$, and  the incident field $u^i(z)=e^{ikz\cdot d}$ with the direction $d\in S^2$ be fixed. The far-field pattern of the rotated scatterer $R^\omega_\vartheta\Om$ satisfies
\begin{align} \label{eqn:rotation}
u^\infty_{R^\omega_\vartheta \Omega}(\hat x;d) \, =\, u^\infty_{\Omega}(R^\omega_{-\vartheta}\hat x;R^\omega_{-\vartheta}d). 
\end{align}
\end{lemma}

\begin{remark}
The translation and rotation formulas in Lemmas~\ref{thm:translation} and~\ref{thm:rotation} exploit the symmetry of the Green’s function and therefore assume a homogeneous background medium.
The approach may nevertheless extend to certain structured inhomogeneous settings, such as piecewise homogeneous media \cite{Liu:2010:DIO}, where such symmetry is locally retained.
\end{remark}

 \subsection{Lipschitz property of far-field patterns in rotation} \label{subsec:ff:stab}

 We investigate the local Lipschitz continuity of the far-field pattern with respect to rotations of a three-dimensional scatterer, estimating its sensitivity to small perturbations of the rotation angle. This property will be exploited in Section~\ref{subsec:bo:angle} to design a prior distribution for Bayesian optimization for recovering the scatterer’s orientation. This in turn enables an analysis of the resulting convergence behavior.
 
The proof below follows the strategy of \cite{Lee:2025:BOA}, leveraging properties of layer potential operators and asymptotic expansions, and is therefore only briefly outlined.

\begin{theorem} \label{thm:angle_stability}
Let $k$ be the wavenumber and $d \in S^2$ be the direction of the incident field. For a fixed $\omega \in S^2$, there exists a constant $C=C(\Om,k,d)$ such that
\begin{align} \label{ineq:angle_stability}
\left| u_{\Omega}^\infty(\hat{x};d) - u_{R^\omega_\vartheta \Omega}^\infty(\hat{x};d) \right|
\leq C |\vartheta|  \quad\mbox{for } |\vartheta|\ll 1.
\end{align}
\end{theorem}

\begin{proof}
Without loss of generality, we set $\omega = (0,0,1) \in S^2$ via a change of axes and drop $\omega$ from $R^\omega_\vartheta$ for notational simplicity.  By \eqref{eqn:ff:orig} and \eqref{eqn:rotation}, we obtain
\begin{align}
& \notag \left| u_{\Omega}^\infty(\hat{x};d) - u_{R_\vartheta \Omega}^\infty(\hat{x};d) \right| =\left| u_{\Omega}^\infty(\hat{x};d) - u_\Omega^\infty(R_{-\vartheta}\hat{x}, R_{-\vartheta}d) \right| \\\label{eqn:estm:ff} 
&  =  \frac{1}{4\pi}\bigg|\int_{\p\Omega} \left[\big(k\hat{x} \cdot n(y) +\eta\big)e^{-ik\hat{x}\cdot y} \phi_d(y) - \big(kR_{-\vartheta}\hat{x} \cdot n(y) +\eta\big)e^{-ikR_{-\vartheta}\hat{x}\cdot y} \phi_{R_{-\vartheta}d}(y) \right] \,ds(y)\bigg|.
\end{align}

Denote the projection onto the $x_1x_2$-plane by $P$ and let $x\in\p\Om$ be given by 
\[x = (\|Px\|\cos s, \|Px\|\sin s, x_3)\]
for some $s\in\RR$. Since
\(R_{-\vartheta}x = (\|Px\|\cos(s-\vartheta), \|Px\|\sin(s-\vartheta), x_3) \),
it holds that for $y=(y_1,y_2,y_3)\in \RR^3$, 
\begin{align}\label{ineq:estm:dot}
 |{x}\cdot y - R_{-\vartheta}{x}\cdot y|\leq \|Px\| \|Py\|_1 |\vartheta|
\end{align}
where $\|Py\|_1=|y_1|+|y_2|$ for the $\ell_1$ norm $\|\cdot\|_1$.
Then using the inequality
$|e^{i\alpha} - e^{i\beta}| \leq |\alpha - \beta|$ for all $\alpha,\beta \in \R$ and \eqref{ineq:estm:dot} with $y=d\in S^2$, we have
\begin{align}\label{ineq:estm:key} 
 \big|e^{ikx\cdot d} - e^{ikx\cdot R_{-\vartheta}d}\big| 
 \leq  k\|Px\| \|Pd\|_1 |\vartheta| \leq \eps(\vartheta),
\end{align}
where \(\eps(\vartheta) := \sqrt{2}kM_{P\Om} |\vartheta|\) and \(M_{P\Omega} = \max\{\|z\| \,:\, z \in \p (P\Omega) \}\), with $P\Omega=\{Pz\,:\, z \in \Omega \}$.

We denote by $\phi_{d}$ and $\phi_{R_{-\vartheta}d}$ the density solutions to \eqref{eqn:density} corresponding to $u^i(x) = e^{ikx\cdot d}$ and $u^i(x) = e^{ikx\cdot R_{-\vartheta}d}$, respectively. 
Since the operator $(\frac{1}{2}I+K_\Om-i\eta S_\Om)^{-1}:C(\p\Om)\rightarrow C(\p\Om)$ is bounded for a coupling parameter $\eta$ (see \cite{Colton:1998:IAE}), \eqref{ineq:estm:key} implies that there exists a constant $C$, depending on $\Om$, such that
\begin{align}\label{ineq:estm:density}
\|\phi_d - \phi_{R_{-\vartheta}d} \|_{L^\infty(\p\Om)} \leq C \eps(\vartheta)\quad\mbox{for }|\vartheta|\ll 1.
\end{align}

By \eqref{eqn:estm:ff} and \eqref{ineq:estm:density}, for $|\vartheta| \ll 1$, we derive 
\begin{gather}\label{ineq:estm:ff}
\left| u_{\Omega}^\infty(\hat{x};d) - u_{R_\vartheta \Omega}^\infty(\hat{x};d) \right| \leq 
\frac{1}{4\pi}\,(I_1 +I_2)
\end{gather}
with 
\begin{align}\notag
I_1 & := \int_{\p\Omega} \Big|\big(k\hat{x} \cdot n(y) +\eta\big)e^{-ik\hat{x}\cdot y} - \big(kR_{-\vartheta}\hat{x} \cdot n(y) +\eta\big)e^{-ikR_{-\vartheta}\hat{x}\cdot y} \Big|\cdot |\phi_d(y)| \,ds(y),\\\label{ineq:estm:I2}
I_2 & :=\int_{\p\Omega}\big|kR_{-\vartheta}\hat{x} \cdot n(y) +\eta\big| \cdot\| \phi_d - \phi_{R_{-\vartheta}d} \|_{L^\infty(\p\Om)} \, ds(y)\leq (k+\eta)\, C|\p \Omega| \eps(\vartheta),
\end{align}
where $|\p\Om|$ is the area of $\p\Om$. 
Then Lemma \ref{lem:angle_stability} below and \eqref{ineq:estm:key} yield
\begin{align*}
&\big|\big(k\hat{x} \cdot n(y) +\eta\big)e^{-ik\hat{x}\cdot y} - \big(kR_{-\vartheta}\hat{x} \cdot n(y) +\eta\big)e^{-ikR_{-\vartheta}\hat{x}\cdot y} \big| \\
 & \leq \big|k(e^{-ik\hat{x}\cdot y} \hat{x} - e^{-ikR_{-\vartheta}\hat{x}\cdot y}R_{-\vartheta} \hat{x}) \cdot n(y)\big| + \big|\eta(e^{-ik\hat{x}\cdot y} - e^{-ikR_{-\vartheta}\hat{x}\cdot y})\big|\\
 & \leq \left(\sqrt{2}k(k\|Py\|_1 +1)|\vartheta| + O(\vartheta^2) \right) +  k\eta\|Py\|_1 |\vartheta|.
\end{align*}
Hence, as $\phi_d\in C(\p\Om)$, we obtain $I_1 \leq C|\vartheta| + O(\vartheta^2)$ for some constant $C=C(\Omega,k,d)$. Combining this with \eqref{ineq:estm:ff} and \eqref{ineq:estm:I2}, we complete the proof. 
\end{proof}

\begin{lemma}\label{lem:angle_stability}
The following inequality holds:
\begin{align*}
\big\|e^{-ik\hat{x}\cdot y} \hat{x} - e^{-ikR_{-\vartheta}\hat{x}\cdot y}R_{-\vartheta} \hat{x} \big\|_2 \leq \sqrt{2}\left(k\|Py\|_1 +1\right)|\vartheta| + O(\vartheta^2)\quad\mbox{for }|\vartheta|\ll 1.
\end{align*}
\end{lemma}
	
\begin{proof}
Writing $\hat{x} = (r\cos s, r\sin s, x_3)$ with $r=\sqrt{x_1^2+x_2^2}$ and some
$s \in \mathbb{R}$, we have
\beq\label{eqn:estm:polar}
\begin{aligned}
\big\|e^{-ik\hat{x}\cdot y} \hat{x} - e^{-ikR_{-\vartheta}\hat{x}\cdot y}R_{-\vartheta} \hat{x} \big\|_2^2
=&\ \left(re^{-ik\hat{x}\cdot y}\cos s - re^{-ikR_{-\vartheta}\hat{x} \cdot y}\cos(s-\vartheta)\right)^2\\
 &+ \left(re^{-ik\hat{x}\cdot y}\sin s- e^{-ikR_{-\vartheta}\hat{x} \cdot y}\sin(s-\vartheta)\right)^2\\
 &+ (x_3e^{-ik\hat{x}\cdot y} - x_3e^{-ikR_{-\vartheta}\hat{x}\cdot y})^2.
\end{aligned}
\eeq
Since  $\cos(s-\vartheta)  = \cos s + \vartheta\sin s + O(\vartheta^2)$ holds by trigonometric identities and the small-angle expansions, it follows that
\begin{align*}
\left|e^{-ik\hat{x}\cdot y}\cos s - e^{-ikR_{-\vartheta}\hat{x} \cdot y}\cos(s-\vartheta)\right|
\leq \, \left|e^{-ik\hat{x}\cdot y} - e^{-ikR_{-\vartheta}\hat{x} \cdot y}\right| + |\vartheta|+ O(\vartheta^2),
\end{align*}
and the same bound holds for $\left|e^{-ik\hat{x}\cdot y}\sin s - e^{-ikR_{-\vartheta}\hat{x} \cdot y}\sin(s-\vartheta)\right|$.
By \eqref{ineq:estm:key}, we also have
$$\left|e^{-ik\hat{x}\cdot y} - e^{-ikR_{-\vartheta}\hat{x} \cdot y}\right| \leq  k\|Py\|_1 |\vartheta|.$$ 
Therefore, from \eqref{eqn:estm:polar} and the identity $r^2+x_3^2 =1$, we prove the lemma. 
\end{proof}

We remark that the bound in Lemma \ref{lem:angle_stability} is consistent with the corresponding two-dimensional result in \cite[Lemma 3.4]{Lee:2025:BOA}. 

\section{Bayesian optimization for orientation recovery} \label{sec:bo} 
Suppose that the scatterer $\Omega$ undergoes a rigid motion in $\R^3$, resulting in the configuration
\begin{equation} \label{eqn:bo:rigid}
   \Omega^* = R_{\vartheta^*}^{\omega^*}\Omega + \tau^*, 
\end{equation}    
where $\tau^* \in \R^3$ is a translation vector, and $R_{\vartheta^*}^{\omega^*}$ denotes the rotation by an angle $\vartheta^* \in (-\pi,\pi]$ about the rotation axis $\omega^* \in S^2$.
Recovering the orientation of $\Omega^*$ in \eqnref{eqn:bo:rigid} involves three degrees of freedom, while the parameter vector $(\vartheta^*,\omega^*)$ belongs to $\RR^4$, which introduces redundancy and motivates a parameterization in $\RR^3$. 

To this end, we adopt the $X$--$Y$--$Z$ (roll--pitch--yaw) fixed-angle representation \cite{Craig:2005:ITR}. 
For any $\omega \in S^2$ and $\vartheta \in (-\pi,\pi]$, there exist angles  $\alpha,\beta,\gamma\in (-\pi,\pi]$ such that $R_{\vartheta}^\omega$ can be expressed as a composition of rotations about the coordinates axes $\hat{e}_1$, $\hat{e}_2$ and $\hat{e}_3$, given by
$$R_{\vartheta}^\omega=R_{(\alpha,\beta,\gamma)}, \quad\mbox{where }  R_{(\alpha,\beta,\gamma)} := R_{\gamma}^{\hat{e}_3} R_{\beta}^{\hat{e}_2}R_{\alpha}^{\hat{e}_1}.$$
In other words,
\begin{align} \label{eqn:rpy2mat}
    R_{\vartheta}^\omega 
    &= \begin{pmatrix}
        \cos\gamma &-\sin\gamma&0\\
        \sin\gamma & \cos\gamma&0\\
        0&0&1
    \end{pmatrix}
    \begin{pmatrix}
        \cos\beta& 0 &\sin\beta\\
        0&1&0\\
        -\sin\beta& 0 & \cos\beta
    \end{pmatrix}
    \begin{pmatrix}
        1&0&0\\
        0&\cos\alpha &-\sin\alpha\\
        0&\sin\alpha & \cos\alpha
    \end{pmatrix}. 
\end{align} 
The angles $\alpha$, $\beta$, and $\gamma$, referred to as roll, pitch, and yaw, correspond to rotations about the $x_1$-, $x_2$-, and $x_3$-axes of a fixed reference frame, respectively. This representation is unique for an arbitrary three-dimensional rotation, except in the singular case $\beta = \pm \pi/2$, known as {\it gimbal lock}. That is, if we denote by $R_{i,j}$ the $(i,j)$ entry of the $3\times3$ matrix $R_\vartheta^\omega$, then away from the singular case (i.e., when $|R_{3,1}|\neq 1$), the angles $(\alpha,\beta,\gamma)$ of $R_\vartheta^\omega$ can be determined from $R_{i,j}$ via 
\begin{align} \label{eqn:mat2rpy}
   \begin{cases}
        \alpha = \operatorname{atan2}\left({R_{3,2}},{R_{3,3}}\right),\\
        \beta = -\operatorname{atan2}\big({R_{3,1}},{\sqrt{R_{1,1}^2+R_{2,1}^2}} \big),\\
        \gamma = \operatorname{atan2}\left({R_{2,1}},{R_{1,1}}\right),
    \end{cases}
\end{align}
where $\operatorname{atan2}$ is the two-argument arctangent.

We now express $\Omega^*$ in \eqnref{eqn:bo:rigid} as 
\begin{align}\label{eqn:bo:target}
    \Omega^* = R_{\bm{\theta}^*}\Omega+\tau^*,\quad\mbox{where }\bm{\theta}^*=(\alpha^*,\beta^*,\gamma^*)\in (-\pi,\pi]^3,
\end{align}
by introducing the angle triple $(\alpha^*,\beta^*,\gamma^*)$ so that $R_{\vartheta^*}^{\omega^*}=R_{(\alpha^*,\beta^*,\gamma^*)}$.

Our objective is to determine both the translation $\tau \in \R^3$ and the orientation parameters $\bm{\theta}=(\alpha,\beta,\gamma)\in(-\pi,\pi]^3$, fixing an incident direction $d$, such that 
\begin{align}\label{eqn:bo:ff} 
    u^{\infty}_{R_{\bm{\theta}}\Omega +\tau}(\hat{x};d) \approx u^\infty_{\Omega^*}(\hat{x};d)
\end{align}
for all measurement directions $\hat{x} \in S^2$. This constitutes one time step in the tracking process. The complete tracking framework will be discussed in Section \ref{sec:framework}, including practical configurations with finitely many measurement directions $\hat{x}$.

We remark that, compared with the two-dimensional setting, this three-dimensional formulation is more challenging due to the increased number of angular parameters and the higher-dimensional search space. We address this challenge via a Bayesian optimization scheme with analytically designed priors. To justify the proposed scheme, we prove posterior consistency for the underlying Bayesian model in Section \ref{subsubsec:bo:analysis}.  

 \subsection{Overview of Bayesian optimization} \label{subsec:bo:bo}

We provide a brief review on Bayesian optimization, which enables efficient optimization of noisy objective functions via adaptive and uncertainty-aware sampling.
    
Bayesian optimization is a sequential framework for efficiently locating the optimizer of an unknown or implicitly defined objective function $f:X\to\mathbb{R}$ by modeling $f$ probabilistically. Starting from a prior distribution over $f$, the method sequentially selects sample points via an acquisition function and updates the posterior using the resulting observations, thereby enabling a substantial reduction in function evaluations compared with standard optimization approaches; see \cite{Brochu:2010:TBO} for further details.

In particular, we consider the problem of minimizing an objective function $f:X\to\mathbb{R}$. Let $P(f)$ denote a prior distribution over $f$, and let $x_i\in X$ represent the \(i\)-th evaluation point for $i=1,2,\dots$. Assume that $f$ can be evaluated at $x_i$, and that the corresponding observation $y_i$ is affected by independent Gaussian noise; that is, 
\begin{align} \label{eqn:bo:observation}
y_i = f(x_i) + \varepsilon_i, \quad \varepsilon_i \overset{\mathrm{i.i.d.}}{\sim} \mathcal{N}(0,\nu),
\end{align}
so that $y_i \sim \mathcal{N}(f(x_i),\nu)$, where $\nu>0$ denotes the noise variance. Then, given a dataset $\mathcal{D}_{1:N}=\{(x_1,y_1),\ldots,(x_N,y_N)\}$, the posterior distribution of \(f\) satisfies
\begin{align} \label{rel:posterior}
P(f | \mathcal{D}_{1:N}) \propto P(\mathcal{D}_{1:N} | f)\,P(f),
\end{align}
where $P(\mathcal{D}_{1:N} | f)$ is a multivariate normal distribution with the mean $(f(x_1),\ldots,f(x_N))$ and the covariance matrix $\nu I_N$.

With a sufficiently large number of observations $N$, an estimate of the optimizer $x^*$ of $f$ can be obtained by minimizing a surrogate model for $f$ derived from the posterior. A Gaussian process prior is commonly adopted, as it enables closed-form expressions for the posterior distribution (see, e.g., \cite{OHagan:1978:CFO,Mockus:1994:ABA}).

To efficiently construct the dataset $\mathcal{D}_{1:N}$, we employ an acquisition function, denoted by ${a}(\cdot)$, which selects the evaluation points $x_i$.
After an initial dataset $\mathcal{D}_{1:i_0}$ is obtained from evaluations of $f$ at a set of predefined (or randomly chosen) points $x_1, \dots, x_{i_0}$, the subsequent points $x_{i_0 +1}, \dots, x_N$ are selected sequentially by maximizing the acquisition function induced by the updated posterior distribution. More precisely, given $\mathcal{D}_{1:i}$ with $i\geq i_0$, the next query point is chosen according to
\begin{align}  \label{eqn:bo:sample}
x_{i+1} = \arg\max_{x\in X} a(x| \mathcal{D}_{1:i}).
\end{align}
The resulting observation $y_{i+1}$ is then collected under the noise model in \eqref{eqn:bo:observation}. 
This procedure is repeated until the full dataset $\mathcal{D}_{1:N}$ is collected. Common choices of acquisition functions include Probability of Improvement \cite{Kushner:1964:NML}, Expected Improvement \cite{Mockus:1978:ABS}, and Lower Confidence Bound (or Upper Confidence Bound for maximization problems) \cite{Cox:1992:SMG,Cox:1997:SDO}. The behavior and performance of these acquisition functions depend on the associated parameters.

\subsection{Inference of the orientation angle}\label{subsec:bo:angle}
In this subsection, we apply the Bayesian optimization framework introduced in Section \ref{subsec:bo:bo} to recover the orientation parameter $\bm{\theta}^*$ of the moving scatterer $\Om^*$ described in \eqnref{eqn:bo:target}. 

As mentioned at the beginning of this section, our goal is to determine the angle parameters $\bm{\theta} = (\alpha,\beta,\gamma)\in(-\pi,\pi]^3$ satisfying \eqnref{eqn:bo:ff}. To achieve this, we define the objective function $f:(-\pi,\pi]^3 \to \R\cup\{+\infty\}$ by
\begin{align} \label{def:bo:objective}
f(\bm{\theta}) := \min_{\tau\in\mathbb{R}^3,\ \|\tau\|<M_\tau} \left\| u^\infty_{R_{\bm{\theta}}\Omega+\tau} - u^\infty_{\Omega^*} \right\|_{L^2(S^2)}, \quad \|\bm{\theta}\|<M_{\bm{\theta}},
\end{align}
and set $f(\bm{\theta})=+\infty$ whenever $\|\bm{\theta}\|\geq M_{\bm{\theta}}$. Here, $M_\tau$, $M_{\bm{\theta}}>0$ are prescribed constants defining the admissible search region. Moreover, $\|\cdot\|$ denotes the Euclidean norm, while $\|\cdot\|_{L^2(S^2)}$ denotes the $L^2$-norm of the far-field patterns on the unit sphere $S^2$. The restriction $\|\bm{\theta}\|<M_{\bm{\theta}}$ confines the optimization procedure to sufficiently small orientation parameters. 
In Section \ref{subsec:framework:method}, we introduce a discrete version of \eqref{def:bo:objective} for computational implementation.

To account for the effect of measurement noise, which is inherent in practical measurement settings, we introduce the noisy far-field data
\begin{align*} 
u^\infty_{\text{meas}} := u^\infty_{\Omega^*} + \eps^{\text{meas}}, 
 \end{align*} 
where $\eps^{\text{meas}}$ denotes the measurement noise. Replacing the exact data in \eqref{def:bo:objective} with $u^\infty_{\text{meas}}$, we define the corresponding noisy objective function $f_{\text{meas}}$ by
\begin{align*} 
f_{\text{meas}}(\bm{\theta}) := \min_{\tau\in\mathbb{R}^3,\ \|\tau\|<M_\tau} \left\| u^\infty_{R_{\bm{\theta}}\Omega+\tau} - u^\infty_{\text{meas}} \right\|_{L^2(S^2)}, \quad \|\bm{\theta}\|<M_{\bm{\theta}},
\end{align*} 
and set $f_{\text{meas}}(\bm{\theta})=+\infty$ when $\|\bm{\theta}\|\geq M_{\bm{\theta}}$. Here, the constants $M_\tau$ and $M_{\bm{\theta}}$ are taken to be identical to those in \eqref{def:bo:objective}. 
We may interpret $f_{\text{meas}}$ as a noisy counterpart of the ideal objective function $f$, namely as $f(\bm{\theta})+\varepsilon$,
where the perturbation $\varepsilon$ arises from measurement noise in the far-field data. This motivates formulating the inverse problem of determining the orientation angle as a Bayesian optimization problem, in which the noisy objective function $f_\text{meas}$ is assumed to follow the observational model in \eqnref{eqn:bo:observation}; see \eqnref{def:bo:yi} below.

The objective function $f(\bm{\theta})$ in \eqnref{def:bo:objective} attains its minimum at the true orientation $\bm{\theta}=\bm{\theta}^*$, where the inner minimizer in $\tau$ coincides with the true displacement $\tau^*$. The limited search regions in $\tau$ and $\bm{\theta}$ reflect the assumption that the scatterer moves only slightly between successive time steps (see Section \ref{subsec:framework:method} for further details). 
We note that, for fixed \(\bm{\theta}\), the inner minimization with respect to $\tau$ can be reformulated using the translation formula \eqref{eqn:translation} as
\begin{align} \label{eqn:bo:translated}
 \min_{\tau\in\mathbb{R}^3,\ \|\tau\|<M_\tau} \Big\|e^{-ik\tau\cdot(\hat{x}-d)}\,u^\infty_{R_{\bm{\theta}}\Omega}(\,\cdot\,;d) - u^\infty_{\Omega^*}(\,\cdot\,;d) \Big\|_{L^2(S^2)}, 
\end{align}
 which can be solved using a gradient-based optimizer such as Adam \cite{Kingma:2014:AMS}.  
 
Although \eqnref{eqn:bo:translated} avoids repeated computations of the far-field pattern for translated configurations $R_{\bm{\theta}}\Om+\tau$, evaluation of $f(\bm{\theta})$ remains computationally expensive due to the embedded minimization in \eqnref{def:bo:objective}. Consequently, Bayesian optimization is a natural choice because of its efficiency in terms of the number of function evaluations. Furthermore, as mentioned above, perturbations induced by measurement noise can be naturally incorporated into the Bayesian optimization framework through an additive Gaussian noise model.

We formulate the Bayesian optimization problem for recovering the orientation parameter $\bm{\theta}^*$ of the moving scatterer $\Om^*$ in \eqnref{eqn:bo:target} as follows.
Let $\bm{\theta}_i$ denote the $i$-th evaluation point and assume that the corresponding noisy observation $y_i$ of $f(\bm{\theta}_i)$ follows the model (see \eqnref{eqn:bo:observation})
\begin{align} \label{def:bo:yi}
y_i = f(\bm{\theta}_i) + \varepsilon_i, \quad \varepsilon_i \overset{\mathrm{i.i.d.}}\sim \mathcal{N}(0,\nu),
\end{align}
where the variance parameter $\nu>0$ represents the level of measurement noise. By taking a prior $P(f)$ and a fixed dataset of $N$ observations $\mathcal{D}_{1:N}=\{(\bm{\theta}_1,y_1),\ldots,(\bm{\theta}_N,y_N)\}$, we compute the posterior distribution via \eqref{rel:posterior}, which yields a posterior estimate of the orientation parameter $\bm{\theta}^*$ associated with the target $\Omega^*$. We describe our choices of the prior distribution and the acquisition function in what follows.  
    
\subsubsection{Prior on the objective function}\label{subsubsec:prior}
A key element for the prior distribution is a candidate optimizer $\hat{\bm{\theta}}$ for the objective function $f$ in \eqref{def:bo:objective}. This candidate is incorporated into the prior distribution and is utilized as the initial point $\bm{\theta}_1$ in the Bayesian optimization process. It is selected as follows.
    
Note that since $|e^{-ik\tau\cdot(\hat{x}-d)}|=1$ for any $\tau \in \R^3$, the modulus of the far-field pattern is invariant under translation and, thus,
\begin{align*}
\left|u_{\Omega^*}^\infty(\hat x;d) \right| = \left|u^\infty_{R_{\bm{\theta}^*}\Omega+\tau^*}(\hat x;d) \right| = \left|e^{-ik\tau^*\cdot (\hat{x}-d)}u^\infty_{R_{\bm{\theta}^*}\Omega}(\hat x;d) \right| = 
\left|u^\infty_{R_{\bm{\theta}^*}\Omega}(\hat x;d) \right|.
\end{align*}
Consequently, if a parameter ${\bm{\theta}}$ coincides with the true orientation $\bm{\theta}^*$, the far-field pattern $u^\infty_{R_{{\bm{\theta}}}\Omega}$ must have the same modulus as $u_{\Omega^*}^\infty$ for all $\hat{x}$. This observation motivates defining a candidate $\hat{\bm{\theta}}$ for $\bm{\theta}^*$ as the minimizer of the discrepancy between the far-field magnitudes: 
\begin{equation} \label{def:bo:thetahat} 
 \hat{\bm{\theta}} := \arg\min_{\left\|\bm{\theta}\right\| <M_{\bm{\theta}}} \left\| \left|u^\infty_{R_{\bm{\theta}}\Omega} \right| - \left| u^\infty_{\text{meas}} \right| \right\|_{L^2(S^2)}.
 \end{equation}
We set $\bm{\theta}_1= \hat{\bm{\theta}}$ as the initial sample for the optimization process. In the noiseless case ($\nu=0$), where $u^\infty_{\text{meas}} = u^\infty_{\Omega^*}$, the discrepancy term in \eqref{def:bo:thetahat} vanishes at $\bm{\theta}  = \bm{\theta}^*$, yielding $\hat{\bm{\theta}}  = \bm{\theta}^*$. We then define a mapping $h$ that quantifies the deviation of $\hat{\bm{\theta}}$ from the true orientation:
\beq\label{def:h} 
 h(\nu) := \hat{\bm{\theta}}  - \bm{\theta}^*.
\eeq
In this expression, $\nu$ serves as the Gaussian noise level associated with the perturbation induced by the measurement noise through the model \eqref{def:bo:yi}.

Next, we specify the prior distribution $P(f)$ for the objective function, incorporating the candidate optimizer $\hat{\bm{\theta}}$. To this end, we examine the asymptotic behavior of the objective function $f$ in \eqnref{def:bo:objective} near the true optimizer $\bm{\theta}^*$. 
By utilizing the translation property, we have
\begin{align*}
\left|u^\infty_{R_{\bm{\theta}} \Omega + \tau}(\hat{x}) - u^\infty_{\Omega^*}(\hat{x})\right| 
=&\left|u^\infty_{R_{\bm{\theta}} \Omega + \tau}(\hat{x}) - u^\infty_{R_{\bm{\theta}^*} \Omega +\tau^*}(\hat{x})\right| \\
= &\left|e^{-ik\tau\cdot(\hat{x}-d)} u^\infty_{R_{\bm{\theta}}\Omega}(\hat{x}) - e^{-ik\tau^*\cdot(\hat{x}-d)} u^\infty_{R_{\bm{\theta}^*}\Omega}(\hat{x})\right| \\
 = &\left|u^\infty_{R_{\bm{\theta}}\Omega}(\hat{x}) - e^{-ik(\tau^*-\tau)\cdot(\hat{x}-d)} u^\infty_{R_{\bm{\theta}^*}\Omega}(\hat{x})\right| \\
= &\left| u^\infty_{R_{\bm{\theta}}\Omega}(\hat{x})-u^\infty_{R_{\bm{\theta}^*}\Omega}(\hat{x})
+ (1-e^{-ik(\tau^*-\tau)\cdot(\hat{x}-d)})u^\infty_{R_{\bm{\theta}^*}\Omega}(\hat{x})\right|.
\end{align*}
It follows that the $L^2$-norm of the difference satisfies
\begin{align}\label{discrep:1}
&\left\|u^\infty_{R_{\bm{\theta}} \Omega + \tau} - u^\infty_{R_{\bm{\theta}^*} \Omega +\tau^*}\right\|_{L^2(S^2)}
=\left\|u^\infty_{R_{\bm{\theta}}\Omega}-u^\infty_{R_{\bm{\theta}^*}\Omega}\right\|_{L^2(S^2)}+r(\bm{\theta},\tau),
\end{align}
where the remainder term $r$ is of order $O(|\tau^*-\tau|)$. 
Since the remainder term vanishes at $\tau=\tau^*$, it follows that the minimum value with respect to $\tau$ in \eqnref{def:bo:objective} is bounded above by the first term on the right-hand side of \eqnref{discrep:1}, i.e., 
\beq\label{discrep:2}
f(\bm{\theta})\leq \|F(\bm{\theta}) -F(\bm{\theta}^*)\|_{L^2(S^2)}\quad\mbox{with }F(\bm{\theta})= u^\infty_{R_{\bm{\theta}}\Omega}.
\eeq
In view of \eqnref{discrep:1} and \eqnref{discrep:2}, if $\|u^\infty_{R_{\bm{\theta}} \Omega + \tau} - u^\infty_{R_{\bm{\theta}^*} \Omega +\tau^*}\|_{L^2(S^2)}$ attains its minimum at a point $(\bm{\theta},\tau)$, then it must hold that
$$\left|r(\bm{\theta},\tau)\right|\leq 2\|F(\bm{\theta}) -F(\bm{\theta}^*)\|_{L^2(S^2)}= O(\|\bm{\theta} - \bm{\theta}^*\|),$$
where the second equality holds by Theorem \ref{thm:angle_stability}.
Therefore, the objective function $f$ can be approximated as
\begin{equation} \label{eqn:bo:objective:approx0}
f(\bm{\theta})=\|F(\bm{\theta}) -F(\bm{\theta}^*)\|_{L^2(S^2)}+O(\|\bm{\theta} - \bm{\theta}^*\|),
\end{equation}

According to Theorem \ref{thm:angle_stability}, $F$ is locally Lipschitz near $\bm{\theta}^*$ with respect to the $L^2(S^2)$-norm. This property implies that, by Rademacher's theorem  \cite{Evans:2015:MTF, Benyamini:2000:GNF}, $F$ is Fr\'echet differentiable almost everywhere in a neighborhood of $\bm{\theta}^*$. Provided that $\bm{\theta}^*$ does not belong to the exceptional set $E_{\bm{\theta}^*}$ of measure zero, $F$ admits the linear approximation
 \begin{equation} \label{eqn:bo:F:approx}
 F(\bm{\theta}) -F(\bm{\theta}^*) = D F(\bm{\theta}^*)(\bm{\theta}-\bm{\theta}^*) + o\left(\left\|\bm{\theta} - \bm{\theta}^*\right\|\right),
\end{equation}
where $ D F(\bm{\theta}^*)$ denotes the Fr\'echet derivative of $F$ at $\bm{\theta}^*$. Consequently, combining \eqref{eqn:bo:objective:approx0} and \eqref{eqn:bo:F:approx} yields
\begin{equation} \label{eqn:bo:objective:approx}
f(\bm{\theta}) =\left\|D F(\bm{\theta}^*)(\bm{\theta}-\bm{\theta}^*)\right\|_{L^2(S^2)} + \delta(\bm{\theta} -\bm{\theta}^*), \quad \left\| \bm{\theta}  \right\| < M_{\bm{\theta}},
\end{equation}
where the remainder term $\delta$ satisfies $|\delta(\cdot)| \leq C\|\cdot\|$ for some constant $C$ depending on $\Omega$ and $\bm{\theta}^*$.

Numerical experiments, including the examples in Section \ref{sec:num}, suggest that the remainder term in \eqnref{eqn:bo:objective:approx} is negligible relative to the first term (see \cite{Lee:2025:BOA} for a visualization of $f$ in the two-dimensional case).
Based on this observation, we employ the linearized approximation in \eqnref{eqn:bo:objective:approx} as the foundation for the prior model.

We construct a prior distribution for $f$ using the candidate angle $\hat{\bm{\theta}}$ as in \eqref{def:bo:thetahat}. 
In particular, we employ a Gaussian process prior for $f(\bm{\theta})$ given by
\begin{align} \label{eqn:bo:GP} 
P(f) = GP(m,k),   
 \end{align} 
where the mean function $m(\bm{\theta})$ and the covariance function $k(\bm{\theta},{\bm{\theta}'})$ are defined by
\begin{align} \label{eqn:bo:GP:mean} 
m(\bm{\theta}) &= \big\| D F(\hat{\bm{\theta}})(\bm{\theta}-\hat{\bm{\theta}}) \big\|_{L^2(S^2)},\\
\label{eqn:bo:GP:covariance}
k(\bm{\theta},\bm{\theta}') &= k(r;l)= \exp\left(-\frac{r}{l}\right)\quad\mbox{with }r = \|\bm{\theta}-\bm{\theta}'\|.
\end{align}
Note that $k$ defines a covariance kernel $k(\bm{\theta},\bm{\theta}') = k(\|\bm{\theta}-\bm{\theta}'\|;l)$. This corresponds to the exponential kernel (equivalently, the Mat\'ern kernel with smoothness parameter $1/2$), with length-scale parameter $l>0$. 
The mean function is chosen based on the approximations \eqref{eqn:bo:objective:approx} using the potential candidate $\hat{\bm{\theta}}$ for $\bm{\theta}^*$, assuming $\bm{\theta}^* \not \in E_{\bm{\theta}^*}$ as described previously. The covariance function $k$ is selected to account for the limited smoothness of $f$ near $\bm{\theta}^*$, as $f$ is not guaranteed to have a continuous gradient. This choice is guided by the approximation of $f$ in \eqref{eqn:bo:objective:approx0} and Theorem \ref{thm:angle_stability}, which implies that $f$ is approximately locally Lipschitz in $\bm{\theta}$ around $\bm{\theta}=\bm{\theta}^*$. 

  \subsubsection{Acquisition function}\label{subsubsec:acquisition}
For the acquisition function $a(\cdot)$, we employ the Lower Confidence Bound (LCB), defined as 
\begin{equation*}
\text{LCB}(\bm{\theta}) = m_{\text{post}}(\bm{\theta}) - \kappa \cdot\sigma_{\text{post}}(\bm{\theta}),
\end{equation*}
where $m_{\text{post}}$ and $\sigma_{\text{post}}$ are the posterior mean and posterior standard deviation induced by the prior and the data observed up to the current query point. The parameter $\kappa\geq0$ controls the trade-off between exploitation and exploration. This explicit parametrization makes it easier to analyze the acquisition function’s behavior compared to alternatives such as Probability of Improvement or Expected Improvement.
We select a new sample point at each iteration according to \eqref{eqn:bo:sample}:
\beq\label{sample:posterior}
\bm{\theta}_{i+1} = \arg\min_{\bm{\theta}} \text{LCB}(\bm{\theta}),
\eeq
where the posterior is conditioned on $\mathcal{D}_{1:i}$.
Here, we minimize rather than maximize the acquisition function, as it is formulated as a lower confidence bound.

We analyze the selection behavior of new sample points given by \eqnref{sample:posterior}, assuming a single data pair $\mathcal{D}_1=(\bm{\theta}_1,y_1)$ with $\bm{\theta}_1=\hat{\bm{\theta}}$, where $\hat{\bm{\theta}}$ is defined by \eqref{def:bo:thetahat}. 
For $\bm{\theta}^*\notin E_{\bm{\theta}^*}$, the posterior distribution is a Gaussian process. Using \eqref{def:bo:yi} together with \eqnref{eqn:bo:GP:mean} and \eqnref{eqn:bo:GP:covariance}, the mean function and covariance kernel are derived as
\begin{align}  
m_{\text{post}}(\bm{\theta}) 
&= m(\bm{\theta}) + \frac{k(\bm{\theta},\hat{\bm{\theta}})} {k(\hat{\bm{\theta}},\hat{\bm{\theta}})+\nu^2} \left(y_1 - m(\hat{\bm{\theta}})\right) \label{eqn:bo:posterior:mean} 
= m(\bm{\theta}) + \frac{k(\bm{\theta},\hat{\bm{\theta}})} {1+\nu^2} \left(f(\hat{\bm{\theta}})+\eps_1\right),\\
\label{eqn:bo:posterior:covariance} 
k_{\text{post}}(\bm{\theta},\bm{\theta}') 
&= k(\bm{\theta},\bm{\theta}') - \frac{k(\bm{\theta},\hat{\bm{\theta}})} {k(\hat{\bm{\theta}},\hat{\bm{\theta}})+\nu^2} k(\hat{\bm{\theta}},\bm{\theta}')
= k(\bm{\theta},\bm{\theta}') - \frac{k(\bm{\theta},\hat{\bm{\theta}})} {1+\nu^2} k(\hat{\bm{\theta}},\bm{\theta}').
\end{align}
In particular, by \eqref{eqn:bo:objective:approx}, we obtain
\beq\notag
\begin{aligned}  
m_{\text{post}}(\bm{\theta}) 
= m(\bm{\theta}) + \frac{e^{-\|\bm{\theta}-\hat{\bm{\theta}}\|/l}}{1+\nu^2}\left(\|DF(\bm{\theta}^*)(\hat{\bm{\theta}}-\bm{\theta}^*)\|_{L^2(S^2)} +\delta(\hat{\bm{\theta}}-\bm{\theta}^*)+ \eps_1 \right). \\
\end{aligned}
\eeq
The posterior variance is given by
\begin{align} \label{eqn:bo:posterior:variance}
\sigma^2_{\text{post}}(\bm{\theta})
 &=k_{\text{post}}(\bm{\theta},\bm{\theta}) = 1-\frac{e^{-2\|\bm{\theta}-\hat{\bm{\theta}}\|/l}}{1+\nu^2}.
\end{align}

Based on the above expressions, we investigate how the parameter $\kappa$ controls exploration through the acquisition function under the posterior conditioned on $\mathcal{D}_1$. In this setting, all prior information is summarized by the single sample $\bm{\theta}_1=\hat{\bm{\theta}}$.
To promote exploration away from $\hat{\bm{\theta}}$, a new sample point $\bm{\theta}$ must satisfy \[\text{LCB}(\bm{\theta}) < \text{LCB}(\hat{\bm{\theta}}),\] that is,
\begin{align*}
 m_{\text{post}}(\bm{\theta}) - \kappa\cdot\sigma_{\text{post}}(\bm{\theta})
< m_{\text{post}}(\hat{\bm{\theta}}) - \kappa\cdot\sigma_{\text{post}}(\hat{\bm{\theta}}).
\end{align*} 
Then, using the notation $h(\nu)=\hat{\bm{\theta}}-\bm{\theta}^*$ defined in \eqref{def:h}, together with \eqref{eqn:bo:posterior:mean}, \eqref{eqn:bo:posterior:variance}, and the inequality $1-e^{-x} \leq x$ for $x\geq 0$, it follows that 
\begin{align} \notag
&m(\bm{\theta}) - \frac{1-e^{-\|\bm{\theta}-\hat{\bm{\theta}}\|/l}}{1+\nu^2}\,\Big(\|DF(\bm{\theta}^*)h(\nu)\|_{L^2(S^2)} +\delta(h(\nu))+ \eps_1 \Big) \\\label{ineq:bo:exploration}
\ &< \kappa\,\frac{1}{\sqrt{1+\nu^2}}\left({\sqrt{\nu^2+2\|\bm{\theta}-\hat{\bm{\theta}}\|/l} -\nu}\right).
\end{align} 
    
We now consider the small-noise regime ($\nu\ll 1$). Assuming $h$ is continuous at $\nu=0$ so that $h(\nu)\to \bm{0}$ as $\nu\to0$, the inequality \eqref{ineq:bo:exploration} reduces to
\begin{align*}
 m(\bm{\theta})  < \kappa\sqrt{\frac{2\|\bm{\theta}-\hat{\bm{\theta}}\|}{l}}.
\end{align*}
By bounding $\|DF(\hat{\bm{\theta}})(\bm{\theta}-\hat{\bm{\theta}})\|_{L^2(S^2)}$ via the operator norm, it follows from \eqref{eqn:bo:GP:mean} that $m(\bm{\theta}) \leq \|D F(\hat{\bm{\theta}})\|_{\RR^3 \to L^2(S^2)}\|\bm{\theta}-\hat{\bm{\theta}}\|$, where $DF(\hat{\bm{\theta}})$ is linear. Hence, in the small-noise regime, the exploration occurs if the new sample $\bm{\theta}$ satisfies
\begin{align} \label{ineq:bo:kappa} 
    \left\|D F(\hat{\bm{\theta}})\right\|_{\RR^3 \to L^2(S^2)}\sqrt{\frac{l}{2}\,\|\bm{\theta}-\hat{\bm{\theta}}\|} < \kappa. 
\end{align}
Consequently, as suggested by the proof of Theorem \ref{thm:angle_stability}, larger targets require larger values of the parameter $\kappa$ to encourage exploration. This reflects the tendency of $\|D F(\hat{\bm{\theta}})\|_{\RR^3 \to L^2(S^2)}$ to increase with target size. 
The resulting strategy is particularly useful when the initial candidate $\hat{\bm{\theta}}$ is unreliable, for instance when it is far from the true rotation angle due to noise.   

\begin{remark} \label{rmk:BO}
In practice, since most available Bayesian optimization modules are designed for maximization, we instead maximize the negated objective function $-f$. Accordingly, we use the Upper Confidence Bound (UCB) as the acquisition function rather than the LCB. Under this formulation, the exploration criterion \eqref{ineq:bo:kappa} remains consistent. 
\end{remark}

\subsection{Posterior consistency of the Bayesian model} \label{subsubsec:bo:analysis}

To provide a theoretical foundation for the proposed optimization scheme, we establish posterior consistency for the surrogate model of the objective function $f$, given a single observation $\mathcal{D}_1$, as the noise level tends to zero. Let $\mu_{\eps}$ denote the distribution of the measurement noise. Posterior consistency in the $\mu_{\eps}$-almost everywhere sense means that, for $\mu_{\eps}$-almost every noise realization, the posterior distribution concentrates on the true function $f$. In other words, the set of noise realizations for which this convergence fails has $\mu_{\eps}$-measure zero. Posterior consistency has been widely studied in Bayesian statistics \cite{Ghosal:2000:CRP} and Bayesian inverse problems \cite{Vollmer:2013:PCB}. 

We recall two classical results from Gaussian process theory that will be used to prove the posterior consistency theorem below.

\begin{lemma}[Dudley’s entropy theorem \cite{Dudley:2016:SWE,Ledoux:1991:PBS}]
    \label{lem:Dudley} 
    Let $\{X_t\}_{t\in T}$ be a centered Gaussian process indexed by a set $T$, with associated pseudo-metric $\rho_X$. Then 
    \begin{align*}
       \mathbb{E}\big(\sup_{t\in T} X_t \big)\lesssim 
       \int_0^{{\rm diam}(T,\rho_X)} \sqrt{\log N(T,\rho_X; r)} \, dr, 
    \end{align*}
    where $\operatorname{diam}(T,\rho_X)$ is the diameter of $T$ with respect to $\rho_X$, and $N(T,\rho_X; r)$ denotes the minimal number of $\rho_X$-balls of radius $r$ needed to cover $T$.
\end{lemma}

\begin{lemma}[Borell--TIS inequality \cite{Adler:2007:RFG}]
\label{lem:Borell} 
Let $\{X_t\}_{t\in T}$ be a centered Gaussian process, almost surely bounded on a set $T$, and define
\begin{align*}
M = \mathbb{E}\big(\sup_{t\in T} X_t\big)\quad\mbox{and}\quad \sigma_t^2 = \mathbb{E}(X_t^2).
\end{align*}
It holds that $M<\infty$ and, for all $u>M$, 
\begin{align*}
P\left(\sup_{t\in T} |X_t| > u \right) 
\leq 2\, \exp\left(-\frac{(u-M)^2}{2\sup_{t\in T}\sigma_t^2} \right).
\end{align*}
\end{lemma} 

With these lemmas at hand, we now present the main theorem in the single-observation setting. The result holds for an arbitrary choice of the prior mean function.
    
\begin{theorem} \label{lem:consistency} 
Let $f$ be the true objective function defined by \eqref{def:bo:objective} for sufficiently small $M_{\bm{\theta}}$. We assume a Gaussian process prior $GP(m,k)$ on $f$, where $m$ is any mean function and $k$ is the covariance kernel defined by \eqref{eqn:bo:GP:covariance}. Given the observation data $\mathcal{D}_{1}=\{(\bm{\theta}_1,y_1)\}$ with measurement model
\begin{align*}
        y_1 = f(\bm{\theta}_1) + \eps_1, \quad \text{ with } \eps_1 \sim \mu_{\eps} = \mathcal{N}(0,\nu),
\end{align*}
let $\hat{f} = m_{\text{post}}$ denote the posterior mean estimator obtained from this Bayesian model.
For a rate function $\eta(\nu) = \nu^\alpha$ with $0<\alpha<1$, it holds that 
\begin{align} \label{eqn:bo:consistency} 
P_{\text{post}}\left(\left\|\hat{f}-f\right\|_\infty > \eta(\nu) \right) \to 0 \quad \mu_{\eps}\text{-a.e.} 
\end{align}
as $\nu\to 0$,
where $P_{\text{post}}$ is the posterior probability measure given $\mathcal{D}_{1}$.
\end{theorem}

\begin{proof}
Assuming the Gaussian process prior $GP(m,k)$ on $f$, the posterior distribution given $\mathcal{D}_{1}$ is also a Gaussian process with a mean function $m_\text{post}$ and a covariance kernel $k_\text{post}$, where $k_\text{post}$ is explicitly given by \eqref{eqn:bo:posterior:covariance}.

Set $\bm{\Theta}= \{\bm{\theta}:\|\bm{\theta}\| < M_{\bm{\theta}}\}$. We define
\begin{align}\label{def:X_theta}
X_{\bm \theta} := m_\text{post}(\bm{\theta}) - f(\bm{\theta}) \quad\mbox{for }\bm{\theta} \in \bm{\Theta}
\end{align}
and denote by $\{X_{\bm \theta}\}_{\bm{\theta} \in \bm{\Theta}}$ the posterior process given $\mathcal{D}_1$, which is the centered Gaussian process with covariance kernel
\(k_{\mathrm{post}}\), i.e.,
\[
\{X_{\bm{\theta}}\}_{\bm{\theta}\in\bm{\Theta}}
\sim GP(0,k_{\mathrm{post}}).
\]

Define the pseudo-metric $\rho_X$ on $\bm{\Theta}$ associated with the process $X_{\bm \theta}$ by
\begin{align*}
        \rho_X(\bm{\theta},\bm{\theta}') := \sqrt{\mathbb{E}(X_{\bm \theta}-X_{\bm \theta'})^2}.
\end{align*}
From \eqref{eqn:bo:posterior:covariance} and \eqnref{eqn:bo:posterior:variance}, it follows that 
 \begin{align*}
        \rho_X(\bm{\theta},\bm{\theta}') &= \sqrt{k_{\text{post}}({\bm \theta},{\bm \theta}) + k_{\text{post}}({\bm \theta}',{\bm \theta}') - 2k_{\text{post}}({\bm \theta},{\bm \theta}')}\\
        &= \sqrt{2 - 2k({\bm \theta},{\bm \theta}') - \frac{\big(k({\bm \theta},\hat{{\bm \theta}}) - k({\bm \theta}',\hat{{\bm \theta}}) \big)^2}{1+\nu^2}}.
\end{align*}
Observe that $\big(k({\bm \theta},\hat{{\bm \theta}}) - k({\bm \theta}',\hat{{\bm \theta}}) \big)^2 = O(\|\bm{\theta} -\bm{\theta}'\|^2)$ by \eqref{eqn:bo:GP:covariance}.
For $\bm{\theta},\bm{\theta}' \in \bm{\Theta}$ with $M_{\bm{\theta}}$ sufficiently small, we have   
 \begin{align*}
        k(\bm{\theta}, \bm{\theta}') = 1-\frac{\|\bm{\theta} -\bm{\theta}'\|}{l} + O(\|\bm{\theta} -\bm{\theta}'\|^2).
\end{align*}
It follows that
    \begin{align} \label{asymp:bo:rho_X}  
        \rho_X(\bm{\theta},\bm{\theta}') \asymp \|\bm{\theta} -\bm{\theta}'\|^{\frac{1}{2}}. 
\end{align}
    
Now, we denote $$M_\nu := \mathbb{E}\big(\sup_{\bm{\theta} \in \bm{\Theta}} X_{\bm\theta}\big)$$
and apply Lemma \ref{lem:Dudley} (Dudley's entropy theorem) to observe
\begin{align} \label{asymp:bo:M_nu:pre}
        M_\nu \lesssim \int_0^{\rm{diam}({\bm\Theta},\rho_X)} 
        \sqrt{\log N(\bm{\Theta}, \rho_X; r)} \,dr \quad \mu_{\eps}\text{-a.e.},
\end{align} 
where \(\operatorname{diam}(\bm{\Theta},\rho_X)\) and \(N(\bm{\Theta},\rho_X;r)\) denote, respectively, the diameter and the covering number of \(\bm{\Theta}\) with respect to \(\rho_X\) (see Lemma \ref{lem:Dudley}). 
Since \(\bm{\Theta}\) is bounded, we have 
\begin{align*}
\text{diam}({\bm\Theta},\rho_X) \leq 2\sup_{\bm{\theta} \in \bm{\Theta}} \sigma_{\text{post}}(\bm{\theta}), 
\end{align*}
and, from \eqref{asymp:bo:rho_X},
\begin{align*}
    N(\bm{\Theta}, \rho_X; r) \lesssim \left(\frac{1}{r^2}\right)^3 = r^{-6}.
\end{align*}
Hence, \eqref{asymp:bo:M_nu:pre} results in
\begin{align} \label{asymp:bo:M_nu}
        M_\nu \lesssim \int_0^{2\sup \sigma_{\text{post}}(\bm{\theta})} 
        \sqrt{\log r^{-1}} \,dr \lesssim \sup \sigma_{\text{post}}  \quad \mu_{\eps}\text{-a.e.} 
\end{align}
Additionally, assuming the smallness of $M_{\bm{\theta}}$,  $\bm{\theta}$ concentrates near $\hat{\bm{\theta}}$ within $\bm{\Theta}$ and \eqref{eqn:bo:posterior:variance} implies  
    \begin{align*}
        \sigma_{\text{post}}(\bm{\theta}) = \sqrt{\frac{\nu^2-2\|\bm{\theta}-\hat{\bm{\theta}}\|/l}{1+\nu^2}} = O(\nu) 
    \end{align*}
as $\nu \to 0$; thus, \eqref{asymp:bo:M_nu} yields
$$M_\nu = O(\nu) \quad \mu_{\eps}\text{-a.e.}$$

In particular, $M_\nu$ is almost surely finite and $\eta(\nu)>M_\nu$, and therefore we can apply Lemma \ref{lem:Borell} (Borell--TIS inequality) to obtain 
\begin{align*}
P\left(\sup_{\bm{\theta} \in \bm{\Theta}} |X_{\bm\theta}| > \eta(\nu) \right) 
& \leq 2\exp\left(-\frac{(\eta(\nu)-M_\nu)^2}{2\sup \sigma_{\text{post}}^2} \right)  \quad \mu_{\eps}\text{-a.e.} 
\end{align*} 
Hence, by using the asymptotic relations together with $\alpha<1$, we conclude that for some constant $c>0$,
\begin{align*}
P\left(\left\|X_{\bm \theta}\right\|_\infty > \eta(\nu) \right) 
& \leq 2\,\exp\left(-c \nu^{2\alpha-2}  \right) \to 0 
\end{align*}
as $\nu \to 0$, where the convergence holds for almost every noise vector drawn from \(\mathcal{N}(0,\nu)\).
In view of \eqnref{def:X_theta}, this proves \eqref{eqn:bo:consistency}.
\end{proof}
        
\begin{remark}
The proof above also applies to other kinds of covariance kernel besides the unsquared exponential kernel. For example, one may use $k_l$ as a Mat\'ern covariance kernel with a half-integer parameter $p+1/2$ ($p=0,1,2$) or the squared exponential kernel, both frequently used as covariance kernels. Then the exponent in \eqref{asymp:bo:rho_X} lies between $1/2$ and $3/2$, yielding the same asymptotic inequality as \eqref{asymp:bo:M_nu}.   
\end{remark}

\section{Full framework for tracking a moving target}\label{sec:framework}

This section introduces a unified framework for identifying the shape and tracking a rigidly moving three-dimensional scatterer. At the initial time ($t=0$), if the shape of the scatterer is unknown, it (more precisely, the boundary surface of the scatterer) is reconstructed from measurement data using a neural network. Once the initial shape has been identified, the position and rotation angle of the scatterer are tracked for $t>0$ by Bayesian optimization. Since the motion is rigid, the reconstructed shape together with its evolving location and orientation completely determines the scatterer. The initial shape identification procedure is described in Section \ref{subsec:framework:shape}, and the subsequent tracking method in Section \ref{subsec:framework:method}. An overview of the tracking framework is provided in Section \ref{subsec:framework:compre}.

 \subsection{Deep learning for shape identification} \label{subsec:framework:shape}
 
We use a fully connected neural network (FCNN) for the shape identification. Let $$\Phi:\R^{n_0}\to\R^{n_L}$$ denote an FCNN with $L+1$ layers, where $L\geq 2$. We set $L=5$ and train $\Phi$ to approximate a target map $\Psi$ from far-field measurements to a parameterization of the boundary surface $\p\Omega$. Here, the input dimension $n_0$ is the number of real-valued components of the far-field data, while the output dimension $n_L$ is the number of parameters describing $\p\Om$. The precise structure of these boundary parameters is given in Section \ref{subsec:num:modeling:ellipsoid}. The network parameters are learned by minimizing a discrepancy-based loss between the outputs of $\Phi$ and those of $\Psi$.

The feasibility of reconstructing $\p\Omega$ from far-field data using FCNNs was established in \cite{Gao:2022:ANN}. Following our previous work \cite{Lee:2025:BOA}, we adopt this approach as a component of the proposed tracking framework. In practice, accurate reconstructions can be obtained using a relatively simple FCNN architecture, provided that the network depth and training parameters are appropriately selected. 

Since the far-field data are complex-valued, we construct $\Phi$ as a real-valued network by concatenating the real and imaginary parts of the measurements. Accordingly, $n_0$ is twice the number of complex-valued input measurements. The layer widths are chosen to be $n_0$, $2n_0$, $3n_0$, $5n_5$, $2n_5$, and $n_5$, where $n_L=n_5$. We employ the SELU activation function to promote self-normalization \cite{Klambauer:2017:SNN}, together with LeCun normal initialization to facilitate stable training. We remark that the choice of activation function and initialization can substantially affect the learning efficiency.

The loss function for $\Phi$ is defined as a weighted $l^2$-loss measuring the discrepancy between the true shape parameters and the real-valued network output. We assign a weight of $2$ to each axis parameter and $3$ to the zeroth-order perturbation parameter $f^1_0$, while assigning weight $1$ to all remaining parameters; see \eqref{eqn:num:bd} and \eqref{eqn:num:bd:f} in Section \ref{subsec:num:modeling:ellipsoid}. These weights reflect their relative importance of the corresponding parameters in characterizing the shape, as well as the greater difficulty of learning $f^1_0$.
Figure \ref{fig:framework:comp}(b) schematically illustrates the overall architecture of \(\Phi\).

The network is trained using the AdamW optimizer with a learning rate of $10^{-4}$ and a mini-batch size of $128$. AdamW is chosen to mitigate the unintended scaling effects of Adam’s adaptive updates and is empirically better suited for SELU-based networks. Training is performed over $5{,}000$ epochs on $20{,}000$ samples, with $16{,}000$ used for training and $4{,}000$ reserved for validation. The dataset is obtained by solving the forward scattering problem \eqref{prob:scattering} with \eqref{eqn:density} and \eqref{eqn:ff:orig} for randomly shaped targets at the initial time, using the shape parameterization described in Section \ref{subsec:num:modeling:ellipsoid}. 
Throughout the tracking procedure presented in Section \ref{subsec:framework:compre}, the network $\Phi$ is assumed to be pretrained offline.

\begin{remark}
Although the overall deep learning architecture follows that of the two-dimensional setting \cite{Lee:2025:BOA}, several modifications are introduced to accommodate the three-dimensional shape parameterization. In particular, the target shape is modeled as a perturbed ellipsoid, whereas in the two-dimensional setting it is modeled as a perturbed ellipse. This change in shape representation also leads to a different weighting strategy in the identification task. Specifically, the semi-axes length parameters of the reference ellipsoid, namely the parameters $a$, $b$, and $c$ in \eqref{eqn:num:bd}, are weighted more heavily than the corresponding parameters in the two-dimensional case, reflecting their increased importance for accurate identification.
\end{remark}

 \subsection{Tracking rigid motion via Bayesian optimization} \label{subsec:framework:method}

We formulate a Bayesian optimization-based tracking framework for recovering the motion of a scatterer from far-field measurements, building on the approach described in Section~\ref{subsec:bo:angle}. Let $\Omega$ denote the scatterer at the initial time, i.e., $\Omega_0=\Omega$. At time $t$, the scatterer $\Omega_t$ is assumed to undergo a rigid motion characterized by a rotation parameter vector $\bm{\theta}_t$ and a translation vector $\tau_t$, so that
\begin{equation}\label{Om_t:def}
\Omega_t = R_{\bm{\theta}_t}\Omega + \tau_t.
 \end{equation}

In three dimensions, rotations are non-commutative, and hence successive rotations specified by multiple angles cannot, in general, be reduced to a single rotation by simply adding the angles. Accordingly, we represent three-dimensional rotations as sequential products of rotation matrices corresponding to roll, pitch, and yaw, in order, and retain the same convention under successive rotations through matrix multiplication. Specifically, a rotation $(\alpha_1,\beta_1,\gamma_1)$ followed by a rotation $(\alpha_2,\beta_2,\gamma_2)$ yields a resultant rotation $(\alpha,\beta,\gamma)$, where $(\alpha,\beta,\gamma)$ is defined by 
\begin{align*} 
\begin{cases}
        \alpha = \operatorname{atan2}\left({R_{3,2}},{R_{3,3}}\right)\\
        \beta = -\operatorname{atan2}\left({R_{3,1}},{\sqrt{R_{1,1}^2+R_{2,1}^2}} \right)\\
        \gamma = \operatorname{atan2}\left({R_{2,1}},{R_{1,1}}\right)
\end{cases}\quad \text{ for } 
 R= \left(R_{\gamma_2}^{\hat{e}_3} R_{\beta_2}^{\hat{e}_2} R_{\alpha_2}^{\hat{e}_1}\right) \left(R_{\gamma_1}^{\hat{e}_3} R_{\beta_1}^{\hat{e}_2} R_{\alpha_1}^{\hat{e}_1}\right),
\end{align*}
provided that $|R_{3,1}| \neq 1$, which excludes the gimbal lock case; see \eqref{eqn:mat2rpy}. 
This formulation naturally captures the sequential rigid motion of the target.

We introduce a discrete-time formulation by sampling time at $t=n\delta$, where $n\in\Z_{\geq 0}$ and $\delta>0$ is a fixed time increment.  
For convenience, we write
\begin{equation}\label{Omega_n:def}
    \Omega_n := \Omega_{n\delta}, \quad \tau_n := \tau_{n\delta}, \quad \bm{\theta}_n := \bm{\theta}_{n\delta},
\end{equation}
and impose the initial conditions $\tau_0=(0,0,0)$ and $\bm{\theta}_0=(0,0,0)$.  
The time increment $\delta$ is assumed to be small, and the admissible ranges of $\tau_n$ and $\bm{\theta}_n$ are taken to be bounded. These constraints reduce the effective search space at each time step, thereby improving the efficiency of the Bayesian optimization procedure.

Unlike in the two-dimensional case, where the observation directions lie on the unit circle $S^1$, the far-field data in three dimensions depend on observation directions on the unit sphere $S^2$.
At each discrete time $t=n\delta$, we assume that the measurements are available at finitely many observation directions $\hat{x}$ on $S^2$. In particular, we suppose that the far-field pattern of $\Om_n$ is sampled at
\begin{equation} \label{def:framework:view}
    \hat{x}_{l,m}=\left(\cos\phi_l\cos\psi_m,\,\sin\phi_l\cos\psi_m,\,\sin\psi_m\right), 
\end{equation}
where 
\begin{equation*}
        \phi_l = \frac{l\pi }{3}, \ l=0,\ldots,5 \quad\mbox{and}\quad
        \psi_m = \frac{m\pi}{3},\  m=-1,0,1. 
\end{equation*}
This yields a total of $18$ measurement points. The resulting far-field data are then assembled into a matrix and represented as
\begin{equation} 
    	u^\infty_n := \begin{pmatrix}
    	u^\infty_{\Omega_{n}}(\hat{x}_{0,-1}) & u^\infty_{\Omega_{n}}(\hat{x}_{0,0}) &u^\infty_{\Omega_{n}}(\hat{x}_{0,1})\\
    	\vdots& \vdots & \vdots\\
    	u^\infty_{\Omega_{n}}(\hat{x}_{5,-1}) & u^\infty_{\Omega_{n}}(\hat{x}_{5,0}) & u^\infty_{\Omega_{n}}(\hat{x}_{5,1})
    	\end{pmatrix} + \eps^{\text{meas}}_n,   \label{def:framework:ff}
\end{equation} 
where $\eps^{\text{meas}}_n$ denotes the measurement noise at time $t=n\delta$.

Given an estimate $(\tau_n,\bm{\theta}_n)$ at time step $n$ and the newly acquired far-field data $u_{n+1}^\infty$, the state at the next time step is determined via the Bayesian optimization procedure presented in Section~\ref{subsec:bo:angle}. Specifically, the optimization is performed in a neighborhood of $(\tau_n,\bm{\theta}_n)$ by minimizing the objective function
\begin{align}\label{eqn:framework:objective}
    f_n(\bm{\theta}) = \min_{\|\tau-\tau_n\| < M_\tau} \left\|u^\infty_{R_{\bm{\theta}}\Omega_0+\tau}-u^\infty_{n+1}\right\|_2, 
    \qquad \|\bm{\theta}-\bm{\theta}_n\| < M_{\bm{\theta}},
\end{align}
where $M_\tau>0$ and $M_{\bm{\theta}}>0$ specify the admissible search bounds (see also \eqref{def:bo:objective} and \eqref{eqn:bo:translated}). Since the far-field data are acquired in the matrix form, the norm $\|\cdot\|_{2}$ denotes the matrix $2$-norm, i.e., the spectral norm, in contrast to the $L^2(S^2)$-norm used in \eqref{def:bo:objective}.

We perform Bayesian optimization using the software module provided in \cite{Nogueira:2014:BOO} under the MIT License, with modifications to accommodate nonzero prior mean functions. In addition, we approximate the derivative term in the prior mean function by forward differences, which converge to the true derivative at differentiable points. 

At each time step, the Bayesian optimization proceeds with ten function evaluations. The procedure is initialized with four informed candidates: the previous orientation estimate $\bm{\theta}_n$ and three candidate orientations for the $(n+1)$-st step. Specifically, the three candidates are chosen as the smallest minimizers defined in \eqref{def:bo:thetahat}, and their average is used as $\hat{\bm{\theta}}$ to mitigate the effect of noise. 
The combination of $\hat{\bm{\theta}}$ and the previous estimate $\bm{\theta}_n$ provides reliable starting points for the optimization, yielding substantially better performance than random initialization. 
The updated rotation $\bm{\theta}_{n+1}$ is then selected as the minimizer of $f_n(\bm{\theta})$ obtained from Bayesian optimization. Accordingly, the translation $\tau_{n+1}$ is chosen as the corresponding minimizer with respect to $\tau$ in \eqref{eqn:framework:objective} at $\bm{\theta} = \bm{\theta}_{n+1}$, that is,
\[
    f_n(\bm{\theta}_{n+1})
    =
    \bigl\|
    u^\infty_{R_{\bm{\theta}_{n+1}}\Omega_0+\tau_{n+1}}
    -
    u^\infty_{n+1}
    \bigr\|_2.
\]
This update rule is repeated sequentially  to recover $(\tau_n,\bm{\theta}_n)$ at each discrete time $t=n\delta$.

 \subsection{Comprehensive framework for tracking a moving scatterer} \label{subsec:framework:compre}

We summarize the overall framework for tracking a moving scatterer using far-field data sampled at discrete times $t=n\delta$, $n=0,1,\ldots$.  
The scatterer $\Omega$ is assumed to be initially centered at the origin and to undergo rigid motion in three dimensions for $t>0$. 

The framework comprises three main components: shape identification (Step $0$), construction of a far-field data dictionary (Step $1$), and tracking of the target’s orientation and location (Steps $2$--$3$) through minimization of the objective function \eqref{eqn:framework:objective}.
If the target shape is known {\it a priori}, Step $0$ may be skipped. In Step $1$, the pitch angle $\beta$ is restricted to remain bounded away from $\pm\pi/2$ to avoid gimbal lock. 
The entire process is depicted in Figure \ref{fig:framework:comp}(a).

\begin{itemize}
\item \textbf{Step 0.}
Reconstruct the reference shape $\Omega_0$ from the far-field data $u_0^\infty$ at the initial time using the trained network in Section \ref{subsec:framework:shape}.

\item \textbf{Step 1.}
Given $\Omega_0$, numerically generate the set of far-field data $u^\infty_{R_{\bm{\theta}}\Omega_0}$ for rotated targets $R_{\bm{\theta}}\Omega_0$, where $\bm{\theta}=(\alpha,\beta,\gamma)$ ranges over a prescribed subset
\begin{equation*}
\bm{\Theta} \subset (-\pi,\pi]\times (-\pi,\pi]\backslash \{\pm\pi/2\} \times (-\pi,\pi].
\end{equation*}
The initial position and orientation are set to $\tau_0=(0,0,0)$ and $\bm{\theta}_0 = (0,0,0)$, respectively.

\item \textbf{Step 2.}
Given the estimate $(\tau_n,\bm{\theta}_n)$ and the far-field data $u^\infty_{n+1}$ at the $(n+1)$th time step, perform Bayesian optimization of the objective function 
\begin{align*}
f_n(\bm{\theta}) = \min_{\|\tau-\tau_n\| < M_\tau} \left\|u^\infty_{R_{\bm{\theta}}\Omega_0+\tau}-u^\infty_{n+1}\right\|_2, 
\qquad \|\bm{\theta}-\bm{\theta}_n\| < M_{\bm{\theta}},
\end{align*}
to find $\bm{\theta}_{n+1}$, as outlined in Section \ref{subsec:framework:method}.

\item \textbf{Step 3.}
Given $\bm{\theta}=\bm{\theta}_{n+1}$, determine $\tau_{n+1}$ as the minimizer of the right-hand side of $f_n(\bm{\theta})$ in Step $2$, to obtain the update $(\tau_{n+1},\bm{\theta}_{n+1})$.
Repeat Steps $2$–$3$ until the tracking procedure terminates.
\end{itemize}

\begin{figure}[H] 
    \centering
    \begin{subfloat}[Flowchart of the overall tracking procedure]
	{\includegraphics[width=0.45\textwidth]{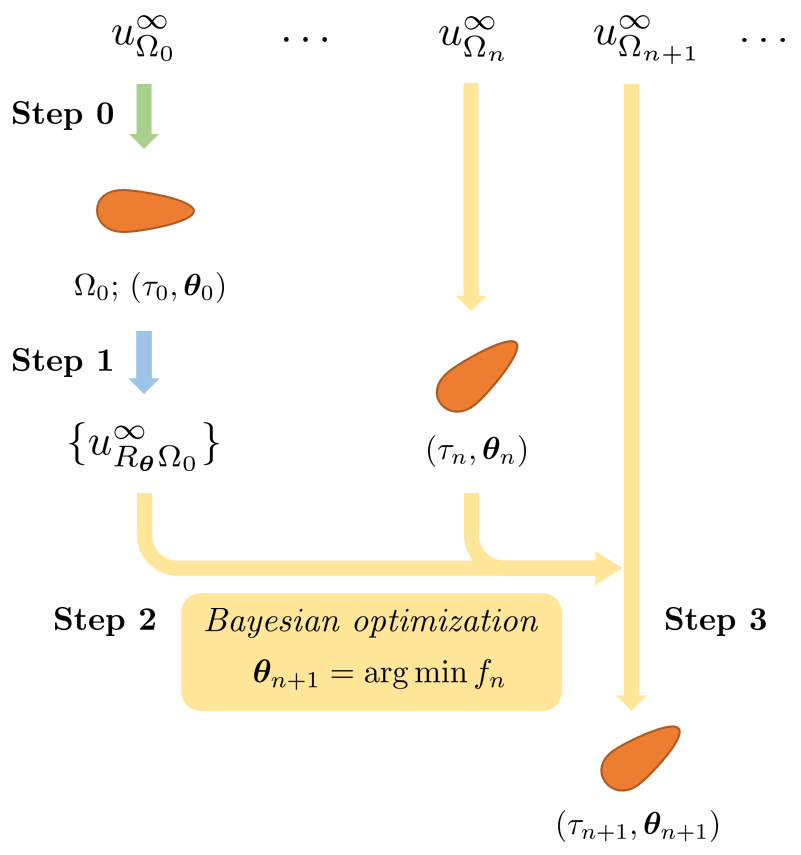}} 
    \end{subfloat} 
    \vfill
    \begin{subfloat}[Neural network architecture for shape identification in Step $0$]
	{\includegraphics[width=0.8\textwidth]{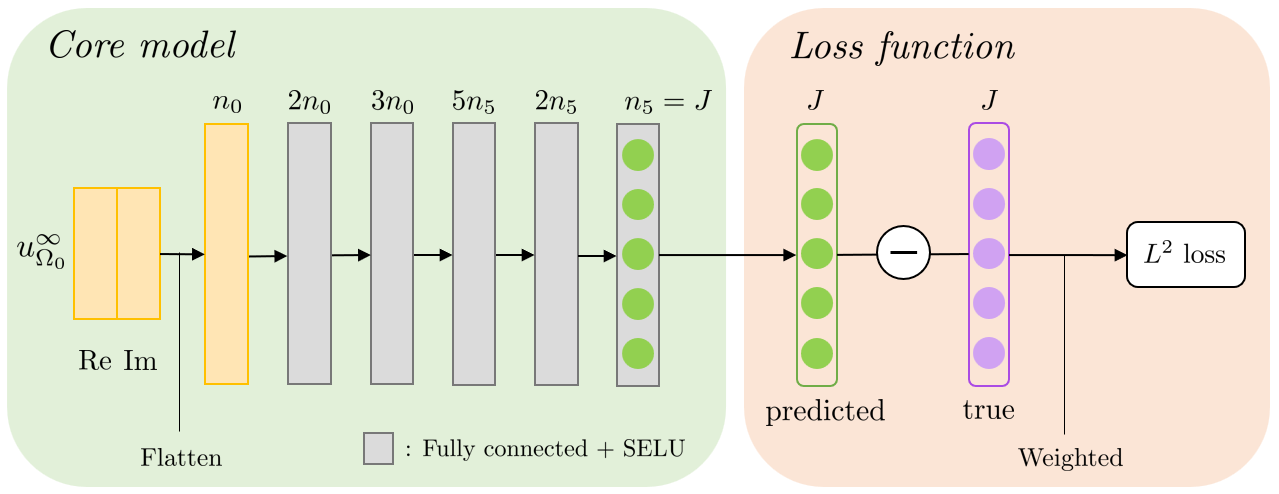}}
    \end{subfloat} 
\caption{\label{fig:framework:comp}
Overview of the proposed three-dimensional tracking framework. 
As illustrated in (a), the framework integrates deep learning-based shape identification with Bayesian optimization-based tracking of a moving scatterer. 
In the neural network architecture shown in (b), the network $\Phi$ used in Step 0 is configured with $L=5$, and $J$ denotes the number of shape parameters defining the target boundary $\partial\Omega_0$.}
\end{figure}

At each time step $t=n\delta$, the scatterer is reconstructed as
\begin{align*} 
\Omega_n  = R_{\bm{\theta}_n}\Omega_0 + \tau_n.
\end{align*} 
This algorithm is not restricted to specific target shapes and trajectories, as supported by Lemma \ref{thm:translation} and Theorem \ref{thm:angle_stability}. As pointed out in \cite{Lee:2025:BOA}, however, the error arising in the Bayesian optimization process during Step 2 may increase for large scatterers. This phenomenon occurs because the local slope of the objective function $f(\bm{\theta})$ near the true angle depends on the the size of the target, as can be inferred from the proof of Theorem \ref{thm:angle_stability}.

\section{Numerical simulations} \label{sec:num}

Following the comprehensive framework for tracking a moving scatterer described in Section~\ref{subsec:framework:compre}, we conduct numerical simulations for three-dimensional tracking. Section~\ref{subsec:num:modeling} and Section~\ref{subsec:num:setup} describe the modeling of the moving targets and the experimental setup, respectively, and Section~\ref{subsec:num:ex} presents the numerical results and discussion.

\subsection{Modeling of a moving target} \label{subsec:num:modeling}

The moving targets in the numerical examples are extended targets at wavenumber $k=1$, whose diameters are larger than $2\pi$. Their shapes are modeled as ellipsoids perturbed by ellipsoidal harmonics \cite{Dassios:2012:EHT}; see Appendix~\ref{appen:ellipsoidal}. The semi-axis lengths of the reference ellipsoid are denoted by $a$, $b$, and $c$, with $a>b>c$. 
Geometric constraints are imposed on the perturbation to ensure that the resulting boundary surface encloses the origin and remains free of self-intersections. These conditions are enforced through parameter restrictions given in Section~\ref{subsec:num:modeling:shape}. The motions of the moving targets are modeled by random processes as described in Section~\ref{subsec:num:modeling:motion}.

 \subsubsection{Perturbed ellipsoids} \label{subsec:num:modeling:ellipsoid}
We model the target at the initial time, namely, $\Omega=\Omega_0$, as an ellipsoid perturbed by ellipsoidal harmonics. A description of three-dimensional objects in terms of ellipsoidal harmonics can be found in \cite{Mademlis:2009:EH3}.
Precisely, the boundary surface $\partial\Omega \subset \mathbb{R}^3$ is parametrized by
\begin{equation} \label{eqn:num:bd}
\begin{aligned}
    \left(\phi, \psi\right) \mapsto
     &\,\left(a\cos\phi\sin\psi,\, b\sin\phi\sin\psi,\, c\cos\psi\right) \\
     &+ \eps \left(bc\cos\phi\sin\psi, \, ac\sin\phi\sin\psi, \, ab\cos\psi\right) \cdot f(\phi,\psi),
\end{aligned}
\end{equation}
where $0\leq\phi<2\pi$, $0\leq \psi\leq \pi$, and the perturbation function $f$ is given by 
\begin{align}\label{eqn:num:bd:f}
f(\phi, \psi)= \sum_{n=0}^{N-1} \sum_{m=1}^{2n+1} f_n^m \, S_n^m(\phi, \psi).
\end{align}
Here, $a$, $b$, and $c$ satisfy $a>b>c>0$ and represent the semi-axis lengths of the reference ellipsoid.  
Moreover, $f_n^m$ are real coefficients, and $S_n^m$ are Lam\'e functions; see Appendix~\ref{appen:ellipsoidal}. The indices satisfy $n=0,\ldots,N-1$ and $m=1,\ldots,2n+1$, where $N$ is a positive integer.

Thus, the target shape is characterized by the semi-axis lengths $a$, $b$, and $c$, the perturbation amplitude $\varepsilon$, the truncation parameter $N$, and the coefficients $f_n^m$. Allowing the semi-axis lengths $(a,b,c)$ to vary yields a class of perturbed ellipsoids that is more expressive than the corresponding class of perturbed spheres.

  \subsubsection{Generation of random target shapes} \label{subsec:num:modeling:shape}

For perturbed ellipsoids of the form \eqref{eqn:num:bd}, we generate admissible target shapes by imposing suitable constraints on the parameters. Without these constraints, inappropriate parameter choices may lead to self-intersecting surfaces or surfaces that do not enclose the origin.

For $N=2$, by the definition of Lam\'{e} functions and the fact that $xy+xz+yz\leq\frac{(x+y+z)^2}{3}$ for real numbers $x,y,z$, we have
\begin{align*}
f(\phi,\psi) &= f_0^1 + f_1^1 h_2h_3\cos\phi\sin\psi + f_1^2 h_1h_3\sin\phi\sin\psi + f_1^3h_1h_2\cos\psi\\
 & \leq |f_0^1| + \left((f_1^1h_2h_3)^2 + (f_1^2h_1h_3)^2 + (f_1^3h_1h_2)^2 \right)^{\frac{1}{2}}\\
& \leq \left|f_0^1\right| + \|f_1\|_\infty \cdot \frac{h_1^2+h_2^2+h_3^2}{\sqrt{3}} 
= \left|f_0^1\right| + \frac{2\|f_1\|_\infty}{\sqrt{3}}(a^2-c^2),
\end{align*}
where $h_1=\sqrt{b^2-c^2}$, $h_2=\sqrt{a^2-c^2}$, $h_3 = \sqrt{a^2-b^2}$, and $f_1=(f_1^1,f_1^2,f_1^3)$.
Since $a>b>c$, the minimum distance from the origin to the reference ellipsoid is $c$. Moreover, the magnitude of the perturbation term is bounded by $\eps ab |f(\phi,\psi)|$.
Thus, the distance from the origin to the boundary of the perturbed ellipsoid is bounded below by
\begin{align*}
    c- \eps ab\left(\left|f_0^1\right|+\frac{2\|f_1\|_\infty }{\sqrt{3}}(a^2-c^2) \right).
\end{align*}
To generate random target surfaces, we impose conditions on the coefficients so that this lower bound remains positive. Namely, we require
\begin{align} \label{ineq:shape:cond} 
    \|{f}_1\|_\infty < \frac{\sqrt{3}}{2(a^2-c^2)}\left( \frac{c}{\eps ab} - \left|f_0^1\right| \right). 
\end{align}
Under this condition, the surface defined by \eqref{eqn:num:bd} encloses the origin.
For example, when $N=2$, $a=8$, $b=5$, $c=4$ and $\eps=0.01$, the choice $|f^1_0|=1$ together with $\|{f}_1\|_{\infty } \leq 0.16$ satisfies \eqref{ineq:shape:cond}.

In the numerical experiments of Section~\ref{subsec:num:ex}, we fix $N=2$ and $\varepsilon=0.01$, and sample the coefficients $f_n^m$ subject to \eqref{ineq:shape:cond}. With a sufficiently small perturbation amplitude, this condition yields boundary surfaces enclosing the origin. 
All generated target shapes are numerically verified to be free of self-intersections.

 \subsubsection{Generation of random target motions} \label{subsec:num:modeling:motion}

We model the rigid motion of the target $\Omega_t$, $t>0$, in \eqref{Om_t:def} by representing its position $\tau_t$ and orientation angle $\bm{\theta}_t$ as three-dimensional Brownian motions. This modeling approach is motivated by the work of Ammari et al.~\cite{Ammari:2013:TMT}, where related stochastic models are used to describe target's motion in two dimensions. 

Let $W_t$ denote a standard three-dimensional Brownian motion.
For each generated motion, the translational and rotational diffusion coefficients $\sigma_v$ and $\sigma_{\bm{\theta}}$ are chosen as time-independent constants and remain fixed throughout the trajectory.

For the translational motion, we define the velocity of the target by 
\begin{align*}
v_t = v_0 + \sigma_v W_t,
\end{align*} 
where $v_0\in\R^3$ is the initial velocity and $\sigma_v>0$ denotes the standard deviation. 
Since $\tau_0=0$, the position of the target is given by
\begin{align*}
\tau_t = \int_0^t v_s ds.
\end{align*} 
For the discrete time step $\delta>0$ and $n \in \Z_{\geq 0}$, let $v_n:= v_{n\delta}$, in accordance with the notation in \eqref{Omega_n:def}. By the properties of Brownian motion, the discrete-time dynamics satisfy
\begin{align} \label{traj:dynamics:v} 
\begin{pmatrix}
v_{n+1} \\ \tau_{n+1} 
\end{pmatrix} = 
\mathbb{F} \begin{pmatrix}
v_{n} \\ \tau_n 
\end{pmatrix} + \mathbb{A}, \quad
\mathbb{F} = \begin{pmatrix}
\mathbb{I}_3 & 0 \\
\delta \mathbb{I}_3 & \mathbb{I}_3 
\end{pmatrix},
\end{align}
where $\mathbb{A}$ is a six-dimensional Gaussian random vector with mean zero and covariance matrix
\begin{align*} 
\Sigma 
= \begin{pmatrix}
\sigma_v^2\,\delta\,\mathbb{I}_3 & \frac{1}{2}\,\sigma_v^2\,\delta^2\,\mathbb{I}_3 \\[2mm]
\frac{1}{2}\,\sigma_v^2\,\delta^2\,\mathbb{I}_3 & \frac{1}{3}\,\sigma_v^2\,\delta^3\,\mathbb{I}_3 \end{pmatrix}. 
\end{align*}

For the rotational motion, we let the orientation evolve according to the Stratonovich stochastic differential equation
\begin{align*} 
dR_{\bm{\theta}_t} =  \widehat{\sigma_{\bm{\theta}} \circ dW_t} R_{\bm{\theta}_t}, 
\end{align*}
where the hat map $\hat{\cdot} :\R^3 \to \mathfrak{so}(3)$ maps a vector in $\RR^3$ to a $3\times 3$ skew-symmetric matrix and $\sigma_{\bm{\theta}}>0$ denotes the standard deviation. The initial orientation angle is denoted by $\bm{\theta}_0\in(-\pi,\pi]^3$ and the corresponding rotation matrix is $R_{\bm{\theta}_0}\in SO(3)$. For the discrete time step $\delta>0$ and $n\in \Z_{\geq 0}$, let $\bm{\theta}_n:= \bm{\theta}_{n\delta}$. Then the orientation is updated according to
\begin{align} \label{traj:theta} 
R_{\bm{\theta}_{n+1}} =  R_{\bm{\rho}}  R_{\bm{\theta}_n},
\end{align}
where $\bm{\rho}$ is a three-dimensional Gaussian random vector with mean zero and covariance matrix $\sigma_{\bm{\theta}}^2 \delta\mathbb{I}_3$.

In all simulations, we take $\sigma_v \in [1,2]$ and $\sigma_{\bm\theta} \in [0.05,0.15]$ to introduce small random perturbations in the target’s motion. For each tracking experiment, a target trajectory is simulated for up to $80$ discrete time steps with $\delta = 0.1$.

 \subsection{Experimental setup} \label{subsec:num:setup}

We describe the experimental setup for the numerical simulations in accordance with the comprehensive tracking framework in Section~\ref{subsec:framework:compre}. The setup comprises a precomputed dataset for shape identification, construction of the far-field data dictionary, the measurement configuration, and the numerical computation of the far-field data.

For target shape identification in Step 0 of the framework, we use a precomputed dataset of $20{,}000$ admissible scatterers. Each sample consists of a target shape and its corresponding far-field data for the initial configuration, represented in the format of \eqref{def:framework:ff}. If the shape is known {\it a priori}, Step 0 is omitted. 

To construct the precomputed dataset, we generate scatterers as extended targets at wavenumber $k=1$ according to the target model described in Section \ref{subsec:num:modeling:shape}. Specifically, the shapes are represented by perturbed ellipsoids defined by \eqnref{eqn:num:bd} and \eqref{eqn:num:bd:f}. We fix the perturbation order at $N=2$ and the perturbation parameter at $\varepsilon=0.01$. The semi-axis parameters $a,b,c \in [4,8]$ are drawn under the constraint $a>b>c$, and the perturbation coefficients $f_n^m$ are sampled to satisfy \eqref{ineq:shape:cond}. In addition, we impose the bounds $|f_0^1|\leq 2$ and $|f_1^m|\leq 0.2$ to avoid excessively oscillatory or distorted boundaries that cannot be resolved reliably by the numerical surface mesh. This parameter choice yields a diverse collection of admissible three-dimensional target shapes.

For the numerical simulations, we consider two types of test targets. The first consists of shapes generated under the same admissibility conditions as those used for the precomputed dataset, while ensuring that they are not contained in the dataset. The second consists of shapes generated outside these admissibility conditions (specifically, with $N=3$). The target motions are generated according to the stochastic update rules \eqref{traj:dynamics:v} and \eqref{traj:theta} described in Section~\ref{subsec:num:modeling:motion}.

To construct the far-field data dictionary in Step 1, we discretize the orientation search space as $\bm{\Theta} = [-60^\circ, 60^\circ]^3$ with a resolution of $1^\circ$ in each angular coordinate $\alpha$, $\beta$, and $\gamma$ (see Section \ref{subsec:framework:method}). 

In both the shape-identification and tracking stages, we use an incident field with the fixed incident direction of $d=(1,0,0)$. During the tracking stage, measurements of the far-field pattern $u_n^\infty$ are acquired at the discrete times $t=n\delta$, $n\geq 0$, with $\delta=0.1$. The far-field sampling configurations are specified in \eqref{def:framework:ff} and \eqref{def:framework:view}. Although the sampling configuration may vary across experiments, the same configuration is used for both shape identification and tracking within each experiment. The full-aperture setting employs $18$ observation directions $\hat{x}$, whereas the upper two-thirds and upper one-third aperture configurations use $12$ and $6$ directions, respectively; see Table~\ref{tab:result:meas} for details.

To model measurement noise, we add complex Gaussian noise with a prescribed signal-to-noise ratio (SNR), scaled based on the mean magnitude of the far-field data.

The far-field data are computed numerically by evaluating \eqref{eqn:ff:orig}, after the integral equation \eqref{eqn:density} is solved with the coupling parameter $\eta = k=1$, to ensure numerical stability \cite{Kress:1985:MCN}. 
All simulations are performed using the boundary element method as implemented in GYPSILAB for MATLAB \cite{Alouges:2018:FBS}. Compared with the two-dimensional case, three-dimensional simulations require sufficiently fine boundary meshes to obtain accurate far-field evaluations.

 \subsection{Reconstruction examples} \label{subsec:num:ex}

Numerical simulations are conducted using the setup described in the preceding subsection.
In all simulations, Bayesian optimization is performed with an angular search radius of $M_{\bm{\theta}}=10^\circ$. Each tracking step requires approximately $30$–$50$ seconds. The higher computational cost relative to the two-dimensional setting is mainly attributable to the use of ten objective function evaluations at each Bayesian optimization step, compared with fewer evaluations in the two-dimensional case.

Figure \ref{fig:ex:full:shape} presents numerical examples of shape identification from full-aperture far-field data at an SNR of $15$dB. The first row compares the true and reconstructed shapes for a target satisfying the geometric constraints described in Section \ref{subsec:num:setup}, whereas the second row shows the corresponding results for a target that does not satisfy these constraints.
Figure \ref{fig:ex:full:traj} illustrates location tracking results obtained from a time-series of far-field data under the upper one-third aperture configuration at an SNR of $5$dB. The target shape is assumed to be known \textit{a priori}. 

\begin{figure}[h!]
    \centering
\includegraphics[width=0.8\textwidth]{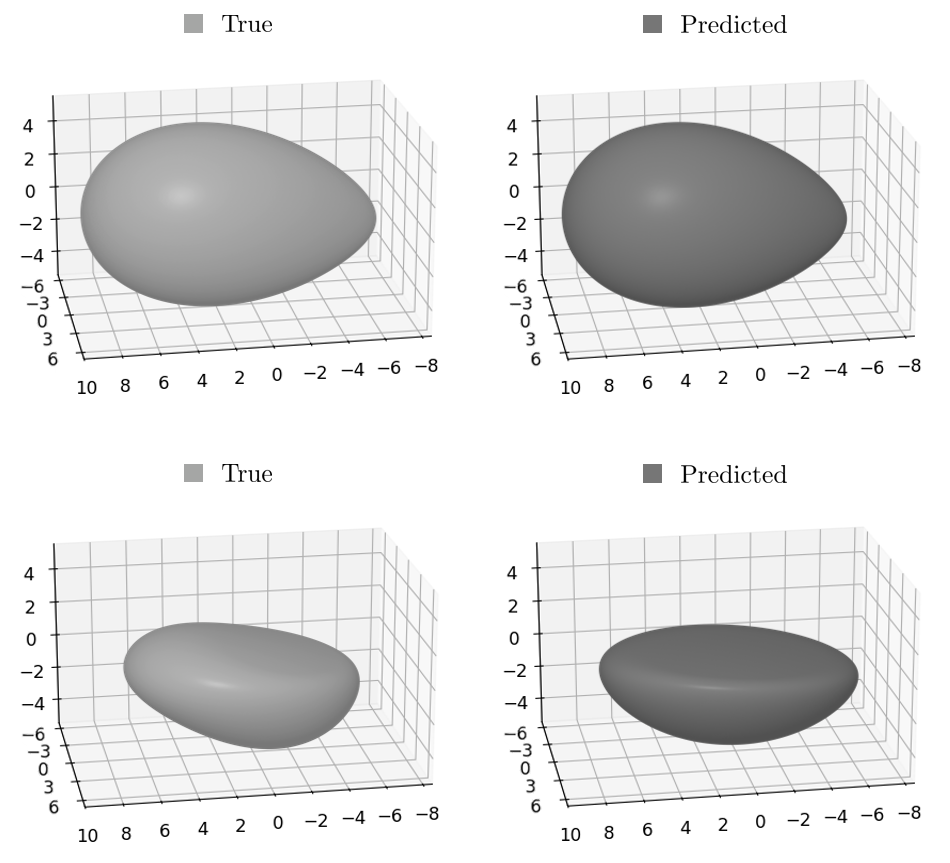} 
\caption{\label{fig:ex:full:shape} Shape identification results obtained from full-aperture far-field data, with $l=0,\ldots,5$ and $m=-1,0,1$, as specified in \eqref{def:framework:view}, at an SNR of $15$dB. The left column shows the true target boundaries, whereas the right column shows the corresponding reconstructed (predicted) boundaries. The target samples in the first and second rows satisfy and do not satisfy, respectively, the geometric constraints described in Section \ref{subsec:num:setup}.}
\end{figure}

\begin{figure}[h!]
\centering
\includegraphics[width=0.95\textwidth]{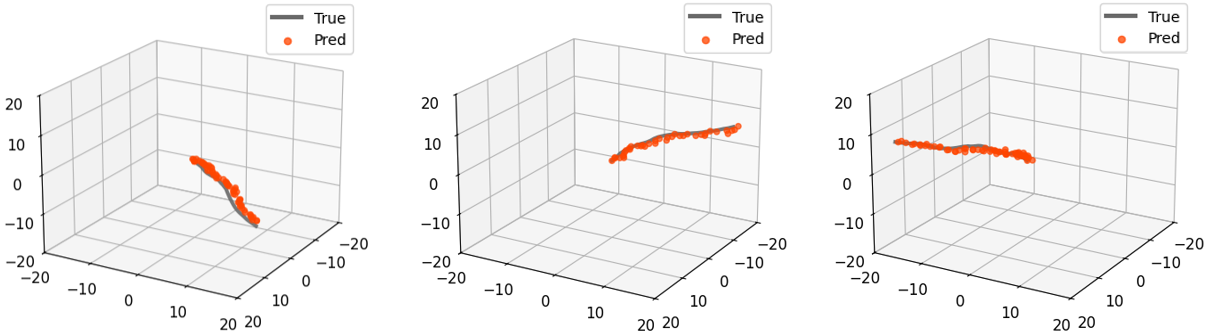} 
\caption{\label{fig:ex:full:traj} Trajectory tracking results obtained from far-field data under the upper one-third aperture, with $l=0,\ldots,5$ and $m=1$, as specified in \eqnref{def:framework:view}, at an SNR of $5$dB. The target shape in this example is assumed to be known {\it a priori}. The gray curve indicates the true trajectory, while the orange markers indicate the estimated locations at the discrete time steps.}
\end{figure}

In what follows, we examine the simultaneous tracking of location and orientation under various measurement configurations. Details of the simulation cases, the corresponding shapes, and the measurement configurations are summarized in Table \ref{tab:result:summary}, Figure \ref{fig:result:shapes}, and Table \ref{tab:result:meas}, respectively.

\begin{table}[h!]
\centering
\caption{\label{tab:result:summary} Summary of the simulation cases and their corresponding result figures. The test targets, Shapes A and B, are shown in Figure~\ref{fig:result:shapes}, and the measurement configurations are listed in Table~\ref{tab:result:meas}.
}
\begin{tabular}{c c c c}
\hline
Prior shape knowledge & Target shape & Measurement type & Result \\
\hline \hline
\multirow{3}{*}{Unknown} & \multirow{3}{*}{Shape A } & Full aperture  & Figure \ref{fig:result:full:unknown} \\
 &  &  Upper two-thirds aperture  & Figure \ref{fig:result:two-thirds:unknown} \\
 &  &  Upper one-third aperture  & Figure \ref{fig:result:one-third:unknown} \\
\hline
Known &  Shape B   & Upper one-third aperture & Figure \ref{fig:result:one-third:known} \\
\hline
\end{tabular}
\end{table} 

\begin{figure}[h!]
    \vskip 1em 
    \centering
    \subfloat[\centering Shape A]{\includegraphics[width=.4\linewidth]{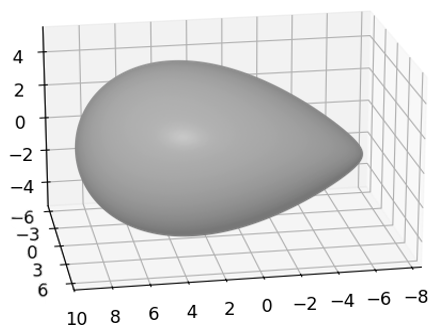}}
    \qquad
    \subfloat[\centering Shape B]{\includegraphics[width=.4\linewidth]{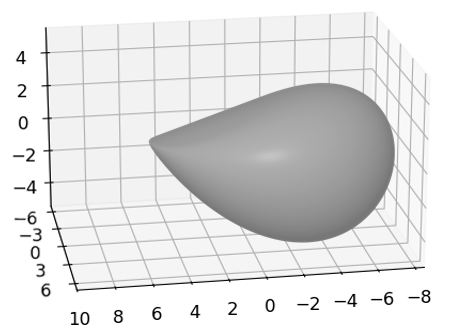}}
    \caption{\label{fig:result:shapes}
    Test targets used in the numerical examples. They are denoted by Shapes~A and~B in Table~\ref{tab:result:summary}. Shape~A satisfies the geometric constraints described in Section~\ref{subsec:num:modeling:shape}, whereas Shape~B does not.} 
\end{figure}

\begin{table}[h!]
\centering
\caption{\label{tab:result:meas} Measurement configurations used in Table \ref{tab:result:summary}. The observation directions are given by $\hat{x}_{l,m}=(\cos\phi_l\cos\psi_m,\sin\phi_l\cos\psi_m,\sin\psi_m)$ with $\phi_l = \frac{l\pi }{3}$ and $\psi_m = \frac{m\pi}{3}$ as defined in \eqref{def:framework:view}.} 
\begin{tabular}{c c}
\hline 
Measurement aperture & Index set \\
\hline\hline
Full  & $l=0,\ldots,5$,\quad $m=-1,0,1$ \\
Upper two-thirds & $l=0,\ldots,5$, \quad $m=0,1$ \\
Upper one-third & $l=0,\ldots,5$, \quad $m=1$  \\
\hline
\end{tabular}
\end{table} 

  \subsubsection{Tracking in full aperture with an unknown shape}

    Figure \ref{fig:result:full:unknown} presents tracking results for an unknown target shape under $15$dB and $10$dB noise levels, corresponding to low- and high-noise regimes, in the full-aperture configuration ($l=0,\ldots,5$, $m=-1,0,1$).
    In both cases, the reconstructed trajectories remain close to the ground truth, with the higher noise level producing noticeably larger deviations.

    \begin{figure}[h!]
    \centering
    \begin{subfloat}[Location tracking under the full-aperture configuration]
	{\includegraphics[width=0.95\textwidth]{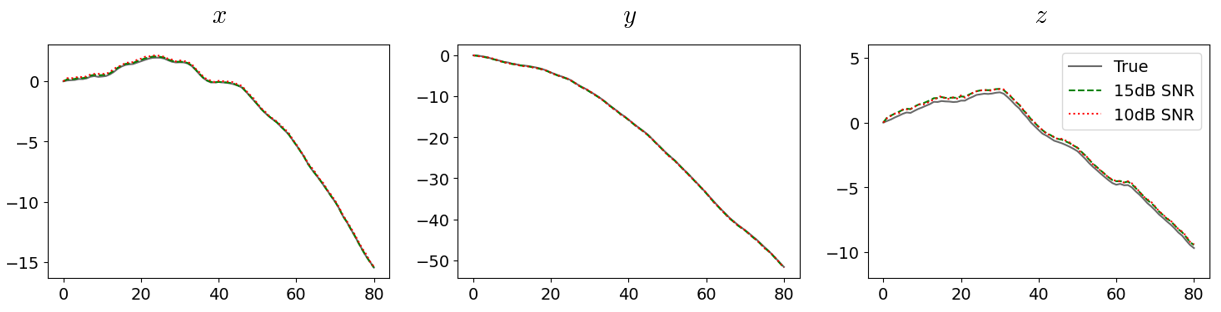}}
    \end{subfloat} 
    \vfill
    \begin{subfloat}[Orientation tracking under the full-aperture configuration]
        {\includegraphics[width=0.95\textwidth]{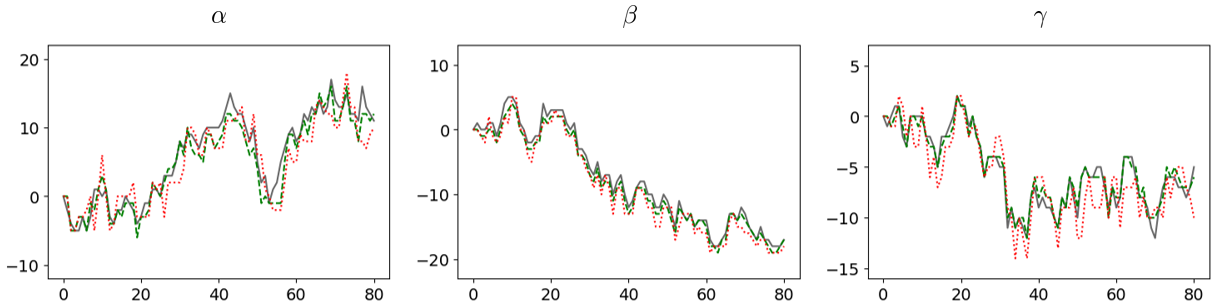}}
    \end{subfloat}
    \caption{\label{fig:result:full:unknown} Tracking results for Shape~A without {\it a priori} information about the target shape, obtained under the full-aperture configuration at $15$dB and $10$dB noise levels. 
    (a) shows the reconstructed target location, with the $x$-, $y$-, and $z$-coordinates displayed from left to right. (b) shows the estimated orientation, with the roll, pitch, and yaw angles $(\alpha,\beta,\gamma)$ displayed from left to right.} 
    \end{figure} 

  \subsubsection{Tracking in limited aperture with an unknown shape}

    We next examine tracking for an unknown target shape using limited-aperture measurements, restricted to the upper two-thirds and upper one-third of the full aperture. Figures \ref{fig:result:two-thirds:unknown} and \ref{fig:result:one-third:unknown} show the corresponding tracking results with different noise levels.
    Compared with the full-aperture case, limiting the measurement aperture leads to reduced robustness in both shape recovery and subsequent tracking. In particular, fewer measurements amplify the sensitivity of the estimated orientation, while higher noise increases the deviation of the recovered trajectories.

\begin{figure}[h!]
    \centering
    \begin{subfloat}[Location tracking under the upper two-thirds aperture configuration]
    {\includegraphics[width=0.95\textwidth]{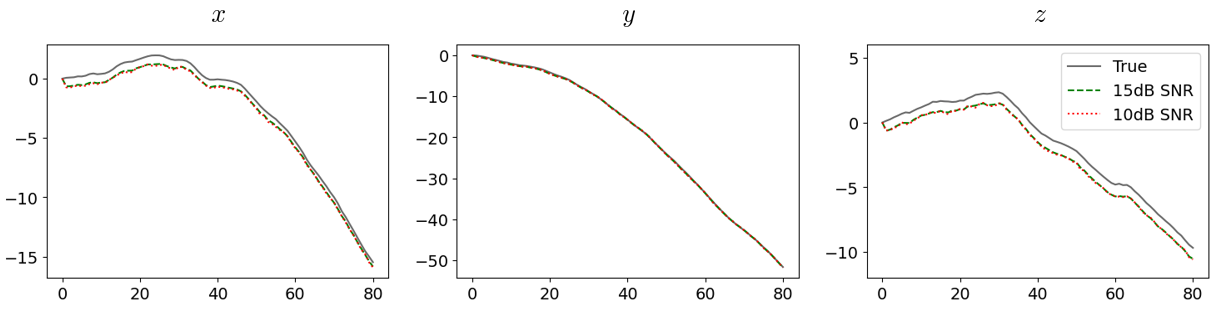}}
    \end{subfloat} 
    \vfill
    \begin{subfloat}[Orientation tracking under the upper two-thirds aperture configuration]
    {\includegraphics[width=0.95\textwidth]{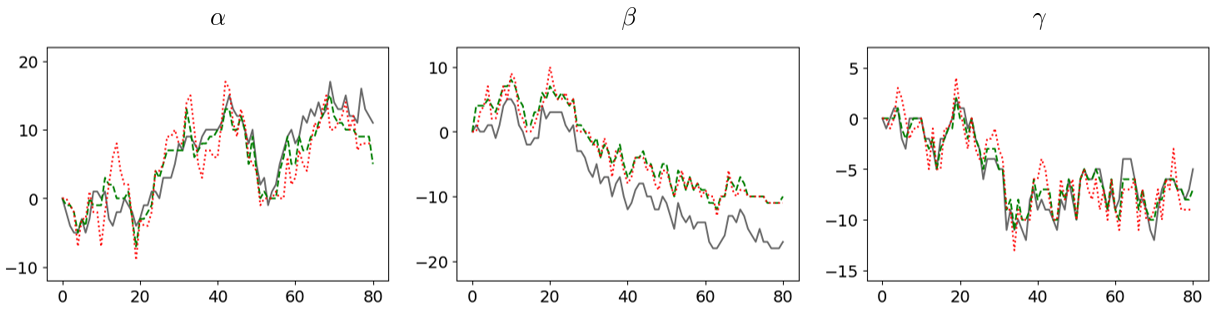}}
    \end{subfloat}
\caption{\label{fig:result:two-thirds:unknown} Tracking results for Shape~A without {\it a priori} information about the target shape, obtained under the upper two-thirds aperture configuration at $15$~dB and $10$~dB noise levels. 
(a) shows the reconstructed target location, with the $x$-, $y$-, and $z$-coordinates displayed from left to right. (b) shows the estimated orientation, with the roll, pitch, and yaw angles $(\alpha,\beta,\gamma)$ displayed from left to right.} 
\end{figure} 

\begin{figure}[h!]
    \centering
    \begin{subfloat}[Location tracking under the upper one-third aperture configuration]
    {\includegraphics[width=0.95\textwidth]{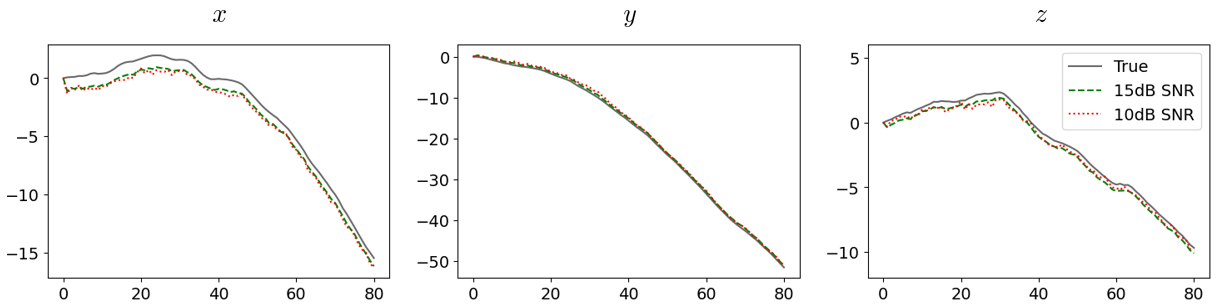}}
    \end{subfloat} 
    \vfill
    \begin{subfloat}[Orientation tracking under the upper one-third aperture configuration]
    {\includegraphics[width=0.95\textwidth]{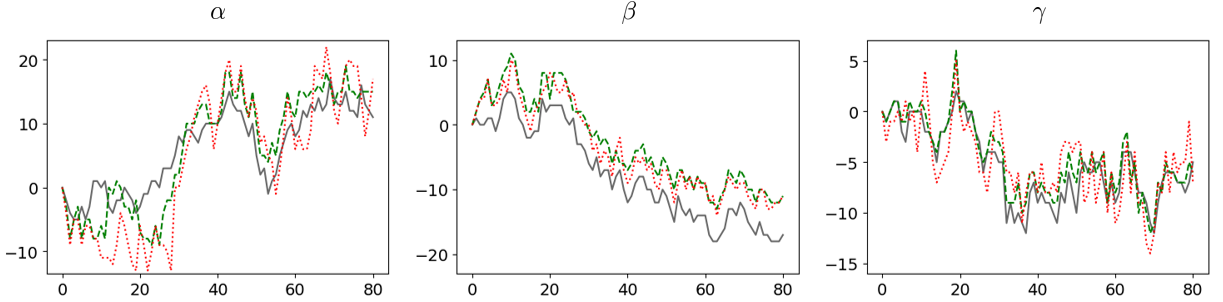}}
    \end{subfloat}
\caption{\label{fig:result:one-third:unknown} Tracking results for Shape~A without {\it a priori} information about the target shape, obtained under the upper one-third aperture configuration at $15$~dB and $10$~dB noise levels. 
(a) shows the reconstructed target location, with the $x$-, $y$-, and $z$-coordinates displayed from left to right. (b) shows the estimated orientation, with the roll, pitch, and yaw angles $(\alpha,\beta,\gamma)$ displayed from left to right.} 
\end{figure}
    
  \subsubsection{Tracking in limited aperture with a known shape}	

    Finally, assuming that the target shape is known \textit{a priori}, we examine tracking with limited-aperture measurements, restricted to the upper one-third of the full aperture. The corresponding results in Figure \ref{fig:result:one-third:known} indicate that accurate tracking of both location and orientation is still achievable, even for geometrically more challenging targets.
    As in the preceding cases, higher noise level leads to larger deviations from the true location and orientation.

\begin{figure}[h!]
        \centering
        \begin{subfloat}[Location tracking under the upper one-third aperture configuration]
    	{\includegraphics[width=0.95\textwidth]{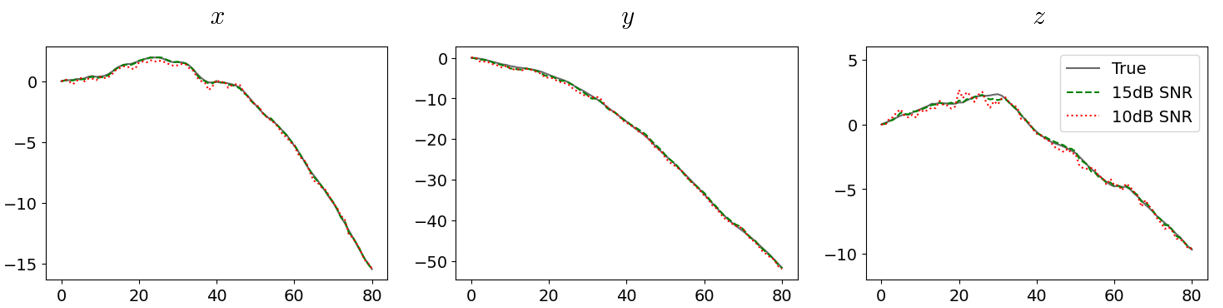}}
        \end{subfloat} 
        \vfill
        \begin{subfloat}[Orientation tracking under the upper one-third aperture configuration]
            {\includegraphics[width=0.95\textwidth]{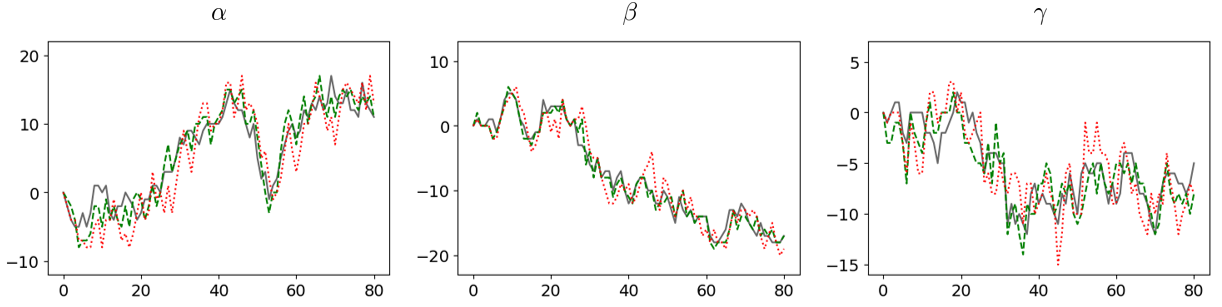}}
        \end{subfloat}
\caption{\label{fig:result:one-third:known} Tracking results for Shape~B with the target shape known {\it a priori}, obtained under the upper one-third aperture configuration at $15$~dB and $10$~dB noise levels. 
(a) shows the reconstructed target location, with the $x$-, $y$-, and $z$-coordinates displayed from left to right. (b) shows the estimated orientation, with the roll, pitch, and yaw angles $(\alpha,\beta,\gamma)$ displayed from left to right.}
\end{figure}

\section{Conclusion} \label{sec:fin}
We have presented a three-dimensional framework for tracking a moving scatterer from far-field data via Bayesian optimization, extending our earlier two-dimensional results. 
To improve the efficiency of the Bayesian optimization in three dimensions, we incorporate analytical characterizations of the far-field pattern under translations and rotations. We also establish posterior consistency for the probabilistic model of the objective function used in the optimization procedure.
Numerical experiments under various measurement apertures and noise levels demonstrate the effectiveness of the method for identifying initially unknown target shapes and subsequently tracking their motion in three dimensions. When the target shape is known {\it a priori} or can be reliably identified, both its location and orientation can be recovered accurately even from restricted and noisy measurements. 
Future work will extend the proposed framework to the tracking of multiple scatterers, building on related developments in this direction \cite{Ohe:2020:RRM,Ohe:2011:RRT,Son:2024:RTT}.

\section*{Acknowledgments}
This work is supported by the National Research Foundation of Korea(NRF) grant funded by the Korea government(MSIT) (RS-2023-00242528, RS-2024-00359109, RS-2025-02303239).

\begin{appendices}

\section{Ellipsoidal harmonics}\label{appen:ellipsoidal}
In this section, we review ellipsoidal harmonics; the reader is referred to \cite{Dassios:2012:EHT} for a thorough discussion. Consider the reference ellipsoid defined by
\[
    E_\text{ref}:\frac{x_1^2}{a^2} + \frac{x_2^2}{b^2} + \frac{x_3^2}{c^2} =1
\]
for $a>b>c>0$. If $\lambda_1$, $\lambda_2$, and $\lambda_3$ are three real roots of the equation
\[
    \frac{x_1^2}{a^2-\lambda} + \frac{x_2^2}{b^2-\lambda} + \frac{x_3^2}{c^2-\lambda} =1,
\]
which are ordered as $-\infty<\lambda_3<c^2< \lambda_2<b^2< \lambda_1<a^2$, we define the ellipsoidal coordinates $(\rho,\mu, \nu)$  by
\[
    \rho =\sqrt{a^2 - \lambda_3}, \quad \mu =\sqrt{a^2 - \lambda_2}, \quad \nu =\sqrt{a^2 - \lambda_1}.
\]
Let $h_1=\sqrt{b^2-c^2}$, $h_2=\sqrt{a^2-c^2}$, and $h_3 = \sqrt{a^2-b^2}$ so that we have $h_1<h_2$ and
\[
    0 \leq \nu^2 \leq h_3^2 \leq \mu^2 \leq h_2^2 \leq \rho^2 <\infty.
\]
Then the family of the associated ellipsoids is given by
\begin{align*} 
    \frac{x_1^2}{\rho^2} + \frac{x_2^2}{\rho^2-h_3^2} + \frac{x_3^2}{\rho^2-h_2^2} =1, \quad \rho \in (h_2,\infty).
\end{align*}
In particular, a point $(x_1,x_2,x_3)$ on the ellipsoid $\rho=\rho_0$ in the first octant is written as
\begin{equation*} 
    x_1 = \frac{\rho_0\mu\nu}{h_2h_3}, \quad
    x_2 = \frac{\sqrt{\rho_0^2-h_3^2}\sqrt{\mu^2-h_3^2}\sqrt{h_3^2-\nu^2}}{h_1h_3},\quad 
    x_3 = \frac{\sqrt{\rho_0^2-h_2^2}\sqrt{h_2^2-\mu^2}\sqrt{h_2^2-\nu^2}}{h_1h_2}.
\end{equation*}

The Lam\'e functions ${E}_n^m$, $m=1,\dots,2n+1$, of degree $n$ are solutions of the Lam\'e equation in ellipsoidal coordinates. For example, $E_0^0(x)=1$, $E_1^1(x) = x$, $E_1^2(x)=\sqrt{|x^2-h_3^2|}$, and $E_1^3(x)=\sqrt{|x^2-h_2^2|}$. The products of two such Lam\'e functions,
\[
    S_n^m(\mu,\nu) = E_n^m(\mu) E_n^m(\nu),
\]
are called the surface ellipsoidal harmonics. These are orthogonal over the surface of any confocal ellipsoid with respect to the weighting function
\[l_\rho(\mu,\nu):= \frac{1}{\sqrt{{\rho}^2-\mu^2}\sqrt{{\rho}^2-\nu^2}}.\]
For instance, at $\rho = a$, which corresponds to the reference ellipse $E_\text{ref}$, we have
\[
    \iint_{E_\text{ref}} S_n^m S_{n'}^{m'} d\Omega = \gamma_n^m \delta_{nn'}\delta_{mm'},
\]
where $d\Omega = h_\mu h_{\nu} l_\rho(\mu,\nu) d\mu d\nu$ and $\gamma_n^m$ denotes the normalization constant. Thus, for $f \in L^1$, we have the expansion
\[
    f(\mu,\nu) = \sum_{n=0}^{N-1} \sum_{m=1}^{2n+1} f_n^m S_n^m(\mu,\nu)
\]
for some $N \in \mathbb{N}$. With this expression, the boundary surface given by the form \eqref{eqn:num:bd} can model the shape of a three-dimensional target.

\end{appendices}

\providecommand{\bysame}{\leavevmode\hbox to3em{\hrulefill}\thinspace}
\providecommand{\MR}{\relax\ifhmode\unskip\space\fi MR }
\providecommand{\MRhref}[2]{%
	\href{http://www.ams.org/mathscinet-getitem?mr=#1}{#2}
}
\providecommand{\href}[2]{#2}

\end{document}